\documentclass[a4paper]{article}
\pdfoutput=1
\usepackage[margin=1in]{geometry}
\usepackage[utf8]{inputenc}
\usepackage[dvipsnames]{xcolor}
\usepackage[bookmarksnumbered,linktocpage,hypertexnames=false,colorlinks=true,linkcolor=NavyBlue,urlcolor=NavyBlue,citecolor=ForestGreen,anchorcolor=green,breaklinks=true,pdfusetitle]{hyperref}
\usepackage{bbm,braket,microtype,amsmath,amssymb,amsthm,amsfonts,dsfont,dynkin-diagrams,mathtools,colortbl,booktabs,graphicx,enumitem,xspace,float,mathdots,ellipsis,caption,subcaption,mleftright,multirow,adjustbox}
\usepackage[T1]{fontenc}
\usepackage{textcomp}
\usepackage[english]{babel}
\usepackage{thmtools} \usepackage[capitalize,nameinlink]{cleveref}
\usepackage[dvipsnames]{xcolor}
\usepackage[boxsize=0.2em]{ytableau} 

\usepackage{tikz}

\newtheorem*{theorem*}{Theorem}
\newtheorem{theorem}{Theorem}[section]\crefname{theorem}{Theorem}{Theorems}
\newtheorem{lemma}[theorem]{Lemma}\crefname{lemma}{Lemma}{Lemmas}
\newtheorem{claim}[theorem]{Claim}\crefname{claim}{Claim}{Claims}
\newtheorem{proposition}[theorem]{Proposition}\crefname{proposition}{Proposition}{Propositions}
\crefname{observation}{Observation}{Observations}
\newtheorem{corollary}[theorem]{Corollary}\crefname{corollary}{Corollary}{Corollaries}
\crefname{conjecture}{Conjecture}{Conjecture}
\theoremstyle{definition}
\newtheorem{definition}[theorem]{Definition}\crefname{definition}{Definition}{Definitions}
\newtheorem{problem}[theorem]{Problem}\crefname{problem}{Problem}{Problems}
\newtheorem{remark}[theorem]{Remark}\crefname{remark}{Remark}{Remarks}
\crefname{example}{Example}{Examples}
\crefname{condition}{Condition}{Conditions}
\numberwithin{equation}{section}

\AddToHook{cmd/appendix/before}{\crefalias{section}{appendix}}

\usepackage{accents}
\DeclareMathAccent{\wtilde}{\mathord}{largesymbols}{"65}
\DeclareMathOperator{\asymR}{\underaccent{\wtilde}{R}}

\DeclareMathOperator{\trank}{R}

\DeclareMathOperator{\slrank}{SR}

\DeclareMathOperator{\GL}{GL}

\DeclareMathOperator{\Lie}{Lie}

\DeclareMathOperator{\rank}{rank}

\DeclareMathOperator{\SL}{SL}

\DeclareMathOperator{\supp}{supp}
\DeclareMathOperator{\SU}{SU}

\DeclareMathOperator{\U}{U}

\DeclareMathOperator{\End}{End}

\newcommand{\Proj}{\mathbb{P}}
\newcommand{\R}{\mathbb{R}}
\newcommand{\C}{\mathbb{C}}
\newcommand{\Z}{\mathbb{Z}}

\newcommand{\N}{\mathbb{N}}
\newcommand{\F}{\mathbb{F}}

\newcommand{\eps}{\varepsilon}

\DeclarePairedDelimiter\abs{\lvert}{\rvert}
\DeclarePairedDelimiter\norm{\lVert}{\rVert}

\DeclarePairedDelimiter\ceil{\lceil}{\rceil}
\DeclarePairedDelimiter\parens{\lparen}{\rparen}

\newcommand{\ot}{\otimes}

\newcommand{\dists}{\mathcal{P}}
\newcommand{\distsnc}{\dists_{\mathrm{lam}}}
\newcommand{\distssing}{\dists_{\mathrm{s}}}
\newcommand{\bipartitions}{\mathcal{B}}

\newcommand{\unit}[1]{\mathsf{U}_{\smash{#1}}}
\newcommand{\unitk}[2]{\mathsf{U}_{#1}^{#2}}
\newcommand{\W}{\mathsf{W}}
\newcommand{\Werner}{\mathsf{Q}} \newcommand{\degenleq}{\trianglelefteq}
\newcommand{\degengeq}{\trianglerighteq}
\newcommand{\asympleq}{\lesssim}
\newcommand{\asympgeq}{\gtrsim}
\DeclareMathOperator{\CR}{CR}
\DeclareMathOperator{\asympCR}{\undertilde{\CR}}
\DeclareMathOperator{\subrank}{Q}
\DeclareMathOperator{\asympQ}{\undertilde{\subrank}}
\DeclareMathOperator{\bordersubrank}{\underline{\subrank}}

\DeclareMathOperator{\capa}{cap}

\newcommand{\mat}[3]{\langle{#1},{#2},{#3}\rangle}

\newcommand{\T}{\mathcal{T}}
\newcommand{\Tnc}{\T_{\mathrm{lq}}}
\newcommand{\semiring}{\mathcal{R}}
\newcommand{\asympspec}{\mathcal{X}}
\newcommand{\kron}{\mathop{\raisebox{-0.25ex}{$\boxtimes$}}}
\usepackage{scalerel}

\allowdisplaybreaks[4]
\usepackage[inline, final]{showlabels}
\renewcommand{\showlabelfont}{\ttfamily\scriptsize}
\def\showlabelsetlabel#1{\raise0ex\hbox{\showlabelfont \textcolor{blue}{#1}}}

\setlist[enumerate]{label=(\roman*),ref=(\roman*)}

\newif\ifanonymous
\anonymousfalse

\usepackage{authblk}
\title{On quantum functionals for higher-order tensors}
\ifanonymous \author{anonymous}
\else
\author[1]{Alonso Botero}
\author[2]{Matthias Christandl} 
\author[2]{Thomas C. Fraser} 
\author[3]{Itai Leigh} 
\author[2]{\texorpdfstring{Harold~Nieuwboer}{Harold Nieuwboer}}
\affil[1]{Departamento de F{\'i}sica, Universidad de los Andes, Bogot{\'a}, Colombia}
\affil[2]{Department of Mathematical Sciences, University of Copenhagen, Copenhagen, Denmark}
\affil[3]{Blavatnik School of Computer Science and AI, Tel-Aviv University, Tel-Aviv, Israel}
\fi
\date{}

\begin{document}

\maketitle

\begin{abstract}
Upper and lower quantum functionals, introduced by Christandl, Vrana and Zuiddam (STOC 2018, J. Amer. Math. Soc. 2023), are families of monotone functions of tensors indexed by a weighting on the set of subsets of the tensor legs. Inspired by quantum information theory, they were crafted as obstructions to asymptotic tensor transformations, relevant in algebraic complexity theory. For tensors of order three, and more generally for weightings on singletons for higher-order tensors, the upper and lower quantum functionals coincide and are spectral points in Strassen's asymptotic spectrum.
Moreover, the singleton quantum functionals characterize the asymptotic slice rank, whereas general weightings provide upper bounds on asymptotic partition rank.
It has been an open question whether the upper and lower quantum functionals also coincide for other cases, or more generally, how to construct further spectral points, especially for higher-order tensors.

In this work, we show that upper and lower quantum functionals generally do not coincide, but that they anchor new spectral points. With this we mean that there exist new spectral points, which equal the quantum functionals on the set of tensors on which upper and lower coincide. The set is shown to include embedded three-tensors and W-like states and concerns all laminar weightings, significantly extending the singleton case.
Moreover, it is shown that these spectral points provide obstructions to asymptotic restriction beyond the previously known spectral points.
\end{abstract}

\section{Introduction}

\subsection{Background and motivation}
The superstar of linear algebra is the linear operation $\psi\colon U\to V$ for $V,U$ finite dimensional spaces over a field $\F$, or in its concrete embodiment (by choosing bases), the matrix $T_{ij}=(T(e_i))_j$. 
A linear operation $\psi$ can also be considered as a bilinear functional $(v,u)\mapsto v^t \psi u\in\F$. 
As linear operations associate two vector spaces, or equivalently, as matrices have two indices, they constitute $2$-tensors. 
Tensors generalize matrices into multidimensional arrays of numbers with $k$ indices, known as $k$-tensors, which represent $k-1$-linear operations or $k$-linear functionals. 
Tensors have found applications in various fields including 
algebraic complexity, additive combinatorics, data compression and quantum information. 
In particular, the case of $3$-tensors has received significant attention, led by the desire to understand the computational aspects of bilinear operations such as polynomial multiplication~\cite{karatsuba1960multiplication} and matrix multiplication~\cite{Strassen1969,WINOGRAD1971lowerbound}. 
The algebraic complexity of these bilinear operations, meaning the number of arithmetic operations in the field (and specifically multiplication) required for their computation, has been shown to be tightly related to a parameter, $\trank(\psi)$, generalizing the rank of matrices, called tensor rank~\cite[\S 4]{Strassen1973vermeidung}. The asymptotic algebraic complexity of these bilinear operations is therefore related to the asymptotic regularization of tensor rank, called asymptotic tensor rank $\asymR(\psi)=\lim_{n\to\infty}(\trank(\psi^{\boxtimes n}))^{1/n}$, where $\psi^{\boxtimes n}$ denotes taking the $n$-fold Kronecker product of $n$ copies of $\psi$.\footnote{
For two $k$-tensors $\psi=\sum_{i_1,\dots,i_k}\psi_{i_1,\dots,i_k}e_{i_1}\otimes \cdots\otimes e_{i_k}$ and $\varphi=\sum_{j_1,\dots,,j_k}\varphi_{j_1,\dots,j_k}e_{j_1}\otimes \cdots\otimes e_{j_k}$, their Kronecker product is the $k$-tensor $\psi\boxtimes \varphi=\sum_{i_1,\dots,i_k,j_1,\dots,j_k}\psi_{i_1,\dots,i_k}\cdot \varphi_{j_1,\dots,j_k}(e_{i_1}\otimes e_{j_1})\otimes\cdots\otimes(e_{i_k}\otimes e_{j_k})$.}

While the correspondence between tensor rank and algebraic complexity has been shown to break for higher order tensors (where $k>3$)~\cite{brand2026bilinearcomplexityworksbreaks}, a more general viewpoint does generalize: restrictions and asymptotic restrictions of tensors are single-use and asymptotic kinds of reductions between $(k-1)$-linear operations appropriate for algebraic complexity, when linear operations on either each input vector or the output vector are free. The tensor rank and its asymptotic regularization in particular describe the minimal number of independent $k$-set-multilinear products needed to compute $\psi$ as a $k$-linear functional, either for one computation of $\psi$ or per-copy of $\psi$ in a concatenated computation of it, $\psi^{\kron n}$. 
In other words, the tensor rank of a tensor $\psi$ and its asymptotic regularization encode the possibility of obtaining $\psi$ from the distinguished family of unit tensors $\unitk{r}{k}=\sum_{i=1}^re_i^{\ot k}$.
Specifically, a $k$-tensor $\psi$ \textit{restricts} to a $k$-tensor $\varphi$, denoted $\psi \geq \varphi$, when there exist linear transformations $M_1,\ldots,M_k$ such that $(M_1\otimes\cdots\otimes M_k)\psi=\varphi$.
From this perspective, the tensor rank $\trank(\psi)$ of a $k$-tensor measures the smallest positive integer $r$ such that $\unitk{r}{k}$ restricts to $\psi$.
The restriction relation between tensors admits of an asymptotic regularization known as \textit{asymptotic restriction}.
A $k$-tensor $\psi$ \textit{asymptotically restricts} to a $k$-tensor $\varphi$, denoted $\psi \asympgeq \varphi$, when there exist sequences of linear transformations $\{M_{1,i},\ldots,M_{k,i}\}_{i\in\N}$ such that $(M_{1,i}\otimes\cdots\otimes M_{k,i})\psi^{\kron n+o(n)}=\varphi^{\boxtimes n}$.\footnotemark{}
Both non-asymptotic and asymptotic restrictions partially order the space of all $k$-tensors by their computational resource.
\footnotetext{While the asymptotic and non-asymptotic restriction partial orders coincide for matrices, 
they are distinct for $k$-tensors when $k\geq3$.
For example, the unit tensor $\unitk{2}{3}$, also known as the Greenberger-Horn-Zeiliger tensor in quantum information, is incomparable to $W_3=e_2\otimes e_1\otimes e_1+e_1\otimes e_2\otimes e_1+e_1\otimes e_1\otimes e_2$ under the restriction order~\cite{dur2000threeQubitsCanBeEntangled}, meaning $\unitk{2}{3}\not\geq W_3$ and $\unitk{2}{3}\not\leq W_3$, while under the asymptotic restriction partial order they become comparable, specifically $\unitk{2}{3}\asympgeq W_3$ and $\unitk{2}{3}\not\asympleq W_3$.}

In additive combinatorics, $k$-tensors can encode (weighted) $k$-uniform hypergraphs. Using such an encoding for $k=3$ and a different generalization of matrix rank, called slice rank, and its asymptotic regularization, Tao reproved Ellenberg's and Gijswijt's bounds for sizes of cap sets~\cite{ellenbergLargeSubsets2017,tao2016capsetSymmetric,taoSawin2016onSliceRank}. The slice rank $\slrank(\psi)$ of a $k$-tensor $\psi$ is the minimal number of slices that sum to it, where a slice is a tensor product $v\otimes u$ where $v\in V_i$ is a vector in one of the local vector spaces and $u\in\bigotimes_{j\in[k]\setminus i}V_j$ is an arbitrary $k-1$-tensor on the rest. On $3$-tensors slice rank identifies with another useful parameter, which differs from it for $k\geq4$: the partition rank. This parameter is defined similarly, but allows in its decompositions more than just slices -- it allows partition tensors, $v\otimes u$ for any $\ell$-tensor $v$ and $(k-\ell)$-tensor $u$, for any $1\leq\ell\leq k-1$. The difference between these parameters is under active investigation, see e.g.~\cite{lampert2026sliceRankITCS}.\footnote{Beyond the single-shot parameters, Lampert and Moshkovitz consider asymptotic questions that are incompatible with the framework of this paper, namely taking $k\to\infty$.}

From the perspective of quantum information theory, normalized complex $k$-tensors coincide with the space of (pure) states for $k$-partite quantum systems, and the restriction preorder between tensors coincides with the resource theory of $k$-partitite \textit{entanglement} under stochastic local operations and classical communication (SLOCC) \cite{bennett2000exact,christandl2024tensor}.
Moreover, as there is a foundational interest in quantum information theory to understand and quantify \textit{genuine} $k$-partite entanglement, i.e., $k$-partite entanglement which cannot be generated using $(k-1)$-partite entanglement \cite{verstraete2002four,schmid2023understanding},
there is a parallel interest in understanding the resource theories of $k$-tensors which are \textit{genuine} in the sense they go beyond those pertaining to $(k-1)$-tensors.

To better understand the space of $k$-tensors under the asymptotic restriction partial order, Strassen developed the theory known as the asymptotic spectrum of tensors~(\cite{strassen1986spectrum} FOCS 1986, \cite{Strassen1987relative,Strassen1988spectrum,Strassen1991degeneration}). 
Central to Strassen's theory is to consider functions $F$ on the space of $k$-tensors, denoted by $\T^{k}$, which preserve their algebraic structure and restriction relations.
Specifically, a function $F\colon \T^k\to\R$ is said to be a restriction \textit{monotone} if $\psi\geq \varphi$ implies $F(\psi)\geq F(\varphi)$, multiplicative over Kronecker products if $F(\psi\boxtimes \varphi)=F(\psi) F(\varphi)$, additive over direct sums if $F(\psi \oplus \varphi) = F(\psi) + F(\varphi)$, and normalized if $F(\unitk{r}{k})=r$. 
These properties of $F$ can be concisely expressed by saying $F\colon \T^k\to\R_{\geq0}$ is a homomorphism of unital preordered semirings.
A function $F$ satisfying these properties is called a universal spectral point, and the space of all universal spectral points, also called the \textit{asymptotic spectrum} of $k$-tensors, is denoted by $\asympspec(\T^k)$. 
The fundamental theorem of Strassen's theory is a duality between asymptotic restriction and universal spectral points: $\psi\asympgeq \varphi$ if and only if $\forall F \in \asympspec(\T^k):\;F(\psi)\geq F(\varphi)$.
One consequence of this duality is a characterization of asymptotic tensor rank as a maximization over the asymptotic spectrum, $\asymR(\psi)=\max_{F\in\asympspec(\T^{k})} F(\psi)$.

Strassen established the duality in greater generality. In particular, for any unital sub-semiring $\T'\subset\T^k$ (i.e. $\T'$ is closed under Kronecker products and direct sums, and includes the unit tensors), denote the set of unital semiring homomorphisms $F'\colon \T'\to\R_{\geq0}$ preserving the restriction partial order by $\asympspec(\T')$ and call them spectral points for $\T'$. Then for any $\psi,\varphi\in\T'$, $\psi\asympgeq \varphi$ iff $\forall F'\in\asympspec(\T'), F'(\psi)\geq F'(\varphi)$. Moreover, any asymptotic spectral point $F'\in\asympspec(\T')$ has an extension $F\in\asympspec(\T^k)$ (i.e. $F|_{\T'}=F'$).

An example of a universal spectral point is a flattening rank: for a non-empty proper subset $S\subset[k]$, the flattening rank $\trank_S(\psi)$ of a $k$-tensor $\psi\in\F^{d_1}\otimes\cdots\otimes \F^{d_k}$, is the matrix rank of $\psi$ as a linear transformation $\psi_S:\otimes_{i\in [k] \setminus S}(\F^{d_i})^*\to\otimes_{i\in S}\F^{d_i}$. 

Strassen also constructed restriction monotones without establishing universality for them~\cite{Strassen1991degeneration}, showing they serve as spectral points for the sub-semiring of \textit{oblique} $3$-tensors, which includes the $2\times 2$ matrix multiplication tensor $\mat222$. 
The construction of these spectral points depends on two parameterized functions, called upper and lower support functionals, denoted $\zeta^\theta$ and $\zeta_\theta$ respectively, and parameterized by a probability distribution over the local vector spaces, $\theta\in\dists([3])$. 
For each parameter $\theta$, the upper support functional upper-bounds the lower support functional $\zeta^\theta(\psi)\geq\zeta_\theta(\psi)$, justifying their names.
Both support functionals are normalized, restriction monotones, and the upper support functional was also shown to be additive $\zeta^\theta(\psi\oplus \varphi) = \zeta^\theta(\psi) + \zeta^\theta(\varphi)$. Meanwhile, the upper support functionals were originally only shown to be sub-multiplicative $\zeta^\theta(\psi\kron \varphi)\leq\zeta^\theta(\psi)\cdot\zeta^\theta(\varphi)$ (and were very recently shown to be multiplicative over the complex numbers~\cite{sakabe2026strassen}), and the lower support functionals are generally super-additive $\zeta_\theta(\psi\oplus \varphi)\geq\zeta_\theta(\psi)+\zeta_\theta(\varphi)$ and super-multiplicative $\zeta_\theta(\psi\kron \varphi)\geq\zeta_\theta(\psi)\cdot\zeta_\theta(\varphi)$. 
Consequently, whenever the upper and lower support functionals agree, they are additive and multiplicative, making them spectral points. Strassen called a tensor $\psi$ \emph{robust} if for every $\theta\in\dists([3])$, $\zeta^\theta(\psi)=\zeta_\theta(\psi)$. We will refer to these tensors as \emph{support-robust}. Strassen showed that oblique tensors are support-robust, making the support functionals spectral points for the sub-semiring of oblique tensors. On the other hand, B{\"u}rgisser showed that over algebraically closed fields the lower support functionals of distributions with full support are strictly super-additive (on some tensors), implying they can not agree with the upper support functionals on all tensors, i.e. not all tensors are support-robust, and more importantly that they are not universal spectral points~\cite[\S 3]{Burgisser1990phd}.

Following the recipe of constructing pairs of upper and lower functionals, Christandl, Vrana and Zuiddam recently developed a notion of upper $F^{\theta}$ and lower $F_{\theta}$ \textit{quantum} functionals for complex tensors (STOC 2018 \cite{cvz18universalSTOC}, J. Amer. Math. Soc. 2023 \cite{cvz23universal}).
The parameter $\theta$ here is a probability distribution over the set of bipartitions of the $k$ local vector spaces,
where a bipartition is complementary disjoint pair of subsets $\{S,\overline{S}\}=\{S,[k]\setminus S\}$ for $\emptyset\neq S\subsetneq[k]$, also denoted by $S|\overline{S}$.
This distribution yields a certain interpolation between the flattening ranks across the bipartitions: if $\theta$ is supported only on a single bipartition, the upper quantum functional evaluates to the flattening rank across this bipartition.
The breakthrough contribution of Christandl, Vrana and Zuiddam was to prove, for all distributions $\theta$ supported on \textit{singleton} bipartitions only\footnote{A singleton bipartition or singleton-vs-the-rest bipartition is a bipartition of the form $S=\{i\}$ and $\overline{S} = [k]\setminus \{i\}$.}, that $F^{\theta}$ and $F_{\theta}$ coincide for all tensors and therefore they constitute \textit{universal} spectral points.
The proof of universality for singleton-supported quantum functionals relies on the notion of a \textit{moment polytope} from geometric invariant theory which establishes a deep correspondence between the representation theory of the group $G = \GL_{d_1} \times \cdots \times \GL_{d_k}$ and singular values of singleton flattenings of tensors~\cite{klyachkoCoherentStatesEntanglement2002,klyachkoQuantumMarginalProblem2004,christandlNonzeroKroneckerCoefficients2007,entanglementPolytope2013,sawickiConvexityMomentumMap2014}. 
Very recently, it was shown by Sakabe, Doğan and Walter~\cite{sakabe2026strassen}, that the singleton-supported lower quantum functionals coincide with Strassen's upper support functional $F_{\theta} = \zeta^{\theta}$. This yields an independent proof that $F_\theta = \zeta^{\theta}$ is a universal spectral point. Their proof relies only on convex analysis on Hadamard manifolds, in particular recent work of Hirai~\cite{hiraiGeneralizedGradientFlows2025}, bypassing the invariant-theoretic upper quantum functional~$F^\theta$.

By contrast, relatively little is known about the upper and lower quantum functionals for non-singleton-supported distributions~$\theta$.
This is the focus of the present paper.
To this end, we let $\dists^k$ denote the set of all probability distributions over bipartitions of $[k]$, while $\distssing^k \subseteq \dists^k$ denotes the subset of singleton-supported distributions.
In their original paper~\cite{cvz23universal}, Christandl, Vrana and Zuiddam managed to additionally show all the properties of the upper and lower quantum functionals, except for their agreement, for a larger subset of distributions -- those whose support satisfies a condition they called \textit{non-crossing}, and we rebrand as \textit{laminar}. The set of distributions with laminar support is denoted by $\distsnc^{k} \subseteq \dists^{k}$.
A set of bipartitions is laminar if for every pair of bipartitions, $S|\overline{S}$ and $T|\overline{T}$, either $S \subseteq T$ or $S \subseteq \overline{T}$ or~$T \subseteq S$ or~$\overline{T} \subseteq S$, i.e., $S$ cannot cross the boundary of the bipartition $T|\overline{T}$ (see \cref{fig:non-crossing-illustration}).\footnote{Note that for the special case of $3$-tensors, where $k=3$, the distinction between singleton-supported, laminar and arbitrary distributions vanishes, i.e., $\distssing^3=\distsnc^3=\dists^3$. For higher order tensors however, where $k\geq4$, these sets of distributions are distinct, i.e., $\distssing^k\subsetneq\distsnc^k\subsetneq\dists^k$.}

\begin{figure}[h]
    \centering
    \tikzset{every picture/.style={line width=0.75pt}} \resizebox{0.9\textwidth}{!}{
\begin{tikzpicture}[x=0.75pt,y=0.75pt,yscale=-1,xscale=1]

\begin{scope}[color=ForestGreen]
\draw  [line width=3] (136,53)-- (136,26) -- (245.5,26) -- (245.5,53) ;

\draw  [line width=3]  (144,53)--(144,34.76) -- (204.5,34.76) -- (204.5,53) ;

\draw  [line width=3]   (213,53) -- (213,34.76) -- (238.5,34.76) -- (238.5,53) ;

\draw [line width=3] (285,53)-- (285,34.76) -- (310.5,34.76) -- (310.5,53) ;

\draw [line width=3]  (253,53)--(253,26) -- (317.5,26) -- (317.5,53) ;

\end{scope}
\begin{scope}[color=BrickRed]
\draw [line width=3] (370,52)-- (370,33.76) -- (429.5,33.76) -- (429.5,52) ;

\draw [line width=3] (393,53) -- (393,26) -- (502.5,26) -- (502.5,53) ;
\end{scope}

\draw (151,47) node [anchor=north west][inner sep=0.75pt] [align=left] {$A$};
\draw (185,47) node [anchor=north west][inner sep=0.75pt] [align=left] {$B$};
\draw (220,47) node [anchor=north west][inner sep=0.75pt]  [align=left] {$C$};
\draw (259,47) node [anchor=north west][inner sep=0.75pt] [align=left] {$D$};
\draw (292,47) node [anchor=north west][inner sep=0.75pt]  [align=left] {$E$};
\draw (374,47) node [anchor=north west][inner sep=0.75pt]  [align=left] {$A$};
\draw (408,47) node [anchor=north west][inner sep=0.75pt]  [align=left] {$B$};
\draw (443,47) node [anchor=north west][inner sep=0.75pt]  [align=left] {$C$};
\draw (478,47) node [anchor=north west][inner sep=0.75pt]  [align=left] {$D$};
\draw (511,47) node [anchor=north west][inner sep=0.75pt]  [align=left] {$E$};

\end{tikzpicture}
}
\caption{The left figure illustrates the \textcolor{ForestGreen}{laminar} family of bipartitions $\textcolor{ForestGreen}{\{AB|CDE,C|ABDE,ABC|DE,E|ABCD,DE|ABC\}}$, while the right figure illustrates a family~$\textcolor{BrickRed}{\{AB|CDE,BCD|AE\}}$ that is \textcolor{BrickRed}{not laminar}.}
    \label{fig:non-crossing-illustration}
\end{figure}

The laminar quantum functionals can provide more information than the singleton-supported quantum functionals. For instance, it is known that the minimum over the singleton-supported quantum functionals is precisely the asymptotic slice rank, whereas the laminar upper quantum functionals yield upper bounds on asymptotic partition rank~\cite[Thm. 5.3, Cor. 5.7]{cvz23universal}. This is connected to the ``one-versus-the-rest'' requirement in slice rank decompositions, whereas partition rank measures rank using rank-$1$ terms across any bipartition.
A concrete example of the separation can be obtained by considering the tensor product of a rank-$2$ unit tensor shared between~$A$ and $B$, and one shared between~$C$ and~$D$.
This has partition rank~$1$ (as it has rank~$1$ across the~$AB|CD$ bipartition), which is detected by the upper quantum functional whose weight is solely on the bipartition~$AB|CD$, but its (asymptotic) slice rank is~$2$, as every singleton quantum functional evaluates to~$2$ on this tensor (it has uniform singleton marginals).
Although this example merely witnesses that the flattening rank across~$AB|CD$ is~$1$, it was recently shown by Lampert and Moshkovitz~\cite{lampert2026sliceRankITCS} that the partition rank of the determinant of~$4 \times 4$ matrices is~$3$, yet its (asymptotic) slice rank is easily seen to be at least~$4$ (uniform singleton marginals), and its flattening rank across any non-singleton bipartition is~$6$.
The laminar upper quantum functionals do not witness this (as the lower quantum functionals achieve their maximum value, hence the same is true for the upper quantum functionals), but it inspires belief that the quantum functionals could separate asymptotic partition rank from asymptotic slice rank; we leave this to future investigation.

Much remains unknown about the laminar quantum functionals. 
In particular, the question of whether any non-singleton-supported (upper or lower) quantum functional that is genuinely $k$-partite\footnote{A universal spectral point for $\ell$-tensors $F\in\asympspec(\T^\ell)$ induces universal spectral points for $k$-tensors, for every $k>\ell$, by grouping local spaces: for every fixed partition $S_1\sqcup \cdots\sqcup S_\ell=[k]$, the induced universal spectral point takes $\psi\in V_1\otimes \cdots\otimes V_k$ and evaluates $F$ on it as an $\ell$-tensor by $\psi\in\bigotimes_{i=1}^\ell(\otimes_{j\in S_i}V_j)$. By a genuinely $k$-partite functional, we mean a functional of $k$-tensors which does not arise in this way.} could agree, and could hence be a universal spectral point, was left unresolved in~\cite{cvz23universal}.
This is the primary focus of the present paper.

\subsection{Results}
We can divide this paper's results into three broad strokes:
    First, we separate the upper and lower quantum functionals parametrized by laminarly-supported distributions that are not singleton-supported, i.e. $\theta \in \distsnc \setminus \distssing$, and in particular, we show that the lower quantum functionals are not universal spectral points.
    Second, we observe that the construction does yield new spectral points for a sub-semiring of tensors, denoted by $\Tnc$, we call the \emph{laminar-quantum-robust} tensors. 
    The sub-semiring $\Tnc$ includes various tensors including iterated matrix multiplication tensors and more generally any graph tensor \cite[Sec.~2.1]{Christandl2019tensorSurgery}.
    We establish that this sub-semiring also contains other tensors, such as the~$\W_4$-tensor from quantum information theory.
    Third, as spectral points for the sub-semiring $\Tnc$, laminar quantum functionals are different than the ones constructed by singleton-supported distributions $\distssing$, even after all possible groupings of local spaces. An extension result of Strassen (\cite[Prop.~3.3]{Strassen1988spectrum}, see also~\cite[Thm.~4.16]{wigdersonAsymptoticSpectraTheorya}) implies that these can be extended into universal spectral points that have yet to be identified.

\paragraph{Laminarly-supported distributions generally are not universal.} We show that any laminarly-supported distribution $\theta$ which has some positive mass $\theta(S|\overline{S})>0$ for a genuine multiparty bipartition $1<|S|,|\overline{S}|$ fails to construct a universal spectral point.
In particular, while singleton-supported quantum functionals were shown to be universal spectral points in~\cite{cvz23universal} by showing their upper and lower quantum functional agree on all tensors, this agreement no longer holds for laminar quantum functionals:
\begin{theorem*}[\cref{thm:upper lower separation general k}]
    For every $\theta\in\distsnc\setminus\distssing$ there exists a tensor $\psi$ for which the upper and lower quantum functionals of $\theta$ disagree: $F_\theta(\psi)< F^\theta(\psi)$.
\end{theorem*}

We give both qualitative and quantitative separations.
The main tool is a classical result due to Kempf and Ness (see \cref{thm:kempf-ness-general}), which provides a rigidity result for states with maximal entropy for each marginal on the local vector spaces: these are unique, up to unitary change of bases, in their~$\SL$-orbit closures.
This allows us to establish tension between having maximal entropy on all single local spaces, and having maximal entropy across non-singleton bipartitions.

The qualitative separation (\cref{prop:psi p gives separation}) is achieved as follows. We define a single-parametric family of $4$-tensors which have both maximal entropies for the marginals on the local spaces and maximal flattening rank across the bipartition $AB|CD$, while its \emph{entropy} across this bipartition is sub-maximal. Thus lower quantum functionals with non-zero weights on $AB|CD$ do not attain the maximal value, while we show that all the corresponding upper quantum functionals do. Embedding these tensors as $k$-tensors gives separations for every $k\geq4$. 

The quantitative separation comes from a new upper bound on the lower quantum functional's evaluation on a tensor $\psi$, which involves the determinant~$\det(\psi_{AB|CD})$ of the flattening of~$\psi$ across the bipartition~$AB|CD$. This bound allows us to show that the separation occurs generically for tensors in~$(\C^d)^{\ot 4}$ (\cref{cor:generic-separation}).
We also exhibit a continuous family~$\mathsf{Q}_\gamma$ of tensors for which the upper bound is tight in~\cref{sec:werner-gamma}.

This new bound also allows us to consider the lower quantum functionals by themselves, showing they are not universal spectral points:
\begin{theorem*}[\cref{prop:general-lower-func-not-spectral}]
    For every $\theta\in\dists\setminus\distssing$, the lower quantum functional~$F_\theta$ is not a universal spectral point: $F_\theta\notin\asympspec(\T^k)$.
\end{theorem*}

We also show that some distributions with non-laminar support over $4$-parties have upper quantum functionals that are not even normalized on the unit $4$-tensor $\unitk{2}{4}$ (\cref{cl:crossing-not-normalised}), so they also are not spectral points\footnote{The definition of the upper quantum functionals becomes unclear when considering distributions with non-laminar supports, because the projectors defining it do not commute. We use the implicit suggestion in \cite[Rem. 3.15]{cvz23universal} to fix an order of the bipartitions, alongside their weights (we do not change the notation and leave the ordering implicit).}.

The quantum functionals were shown to be non-trivial universal spectral points when the distributions are singleton-supported.  
All of the above suggests that the natural way to extend them by considering other distributions might not be a source for yet newer universal spectral points.

\paragraph{Laminar-quantum-robust tensors.} Despite the above non-universality result for laminar quantum functionals, one might wonder whether they remain spectral points for some strict sub-semiring of tensors. 
This would render laminar quantum functionals potentially useful for characterizing the asymptotic properties of tensors in this strict subfamily. 
If so, this would put laminar quantum functionals in a similar position to Strassen's original support functionals, which were shown to be spectral points merely for a sub-semiring including the motivating matrix multiplication tensor\footnote{As mentioned above, very recently the upper support functionals were shown to be universal for complex tensors after all~\cite{sakabe2026strassen}. This does not change the fact that they were developed as spectral points for a sub-semiring in order to investigate the matrix multiplication tensor.}. 
In this paper, we establish that this is indeed the case. 

Denote by $\Tnc^k\subset\T^k$ the set of $k$-tensors on which all the laminarly-supported distributions have agreeing upper and lower quantum functionals: $\Tnc^k = \{ \psi \in \T^k \mid \forall \theta \in \distsnc, \, F^\theta(\psi) = F_\theta(\psi) \}$ (\cref{def:Tnck}). Using the sub/super-additivity/multiplicativity properties of the quantum functionals together with the inequality relation $F_{\theta}(\psi) \leq F^{\theta}(\psi)$ between them, we show the subset $\Tnc^k$ is closed under direct sums and the Kronecker product.
As $\Tnc^k$ also includes the unit tensors, $\Tnc^k$ constitutes a unital sub-semiring~\cref{prop:Tnck is a semiring}. 
Following Strassen's naming conventions~\cite{Strassen1991degeneration} we refer to the set $\Tnc^k$ as the semiring of laminar-quantum-robust tensors.

In the case of $3$-tensors the definition of $\Tnc^3$ is degenerate: $3$ parties have only singleton-supported distributions and these are universal spectral points by~\cite{cvz23universal}. Thus, the laminar-quantum-robust semiring for $3$-tensors is just all the $3$-tensors, i.e., $\Tnc^3 = \T^3$. 
This property of $3$-tensors is to be contrasted with $k$-tensors for any $k \geq 4$, where \cref{thm:upper lower separation general k} establishes $\Tnc^k \subsetneq \T^k$ is a proper subset.

This triviality for $3$-tensors can be used for higher-order tensors by considering embeddings of them as higher-order tensors $\psi\mapsto\tilde\psi:=\psi\otimes e_1^{\otimes k-3}$ (\cref{def:embedded tensor}). Taking direct sums of Kronecker products of such tensors constructs a rich sub-semiring we denote by $\T_3^k\subset\Tnc^k$. More generally the embedding of any $\psi\in\Tnc^\ell$, for $\ell< k$, is in $\Tnc^k$ (\cref{lem:embedded-tensors}). $\Tnc^k$ trivially includes Absolutely Maximally Entangled (AME) tensors, which are well-studied in quantum information theory, as well as the determinant tensor~$\det_k = e_1 \wedge \dotsb \wedge e_k \in (\C^k)^{\otimes k}$  (\cref{remark:AME-states-are-in-Tnc}).
As we prove in this paper, $\Tnc^4$ also includes the tensor $\W_4=e_2 \ot e_1 \ot e_1 \ot e_1 + e_1 \ot e_2 \ot e_1 \ot e_1 + e_1 \ot e_1 \ot e_2 \ot e_1 + e_1 \ot e_1 \ot e_1 \ot e_2$ (\cref{thm:W4-non-crossing-agree}), which is not in $\T_3^4$, showing the embeddings of $3$-tensors do not generate the whole laminar-quantum-robust sub-semiring $\T_3^k \subsetneq \Tnc^k$.

We summarize our results into the following theorem, as well as  \cref{fig:venn-semirings}:

\begin{theorem*}
    For any $k\geq4$, the laminar-quantum-robust sub-semiring is strictly contained in the semiring of all $k$-tensors and strictly contains the semiring generated by embeddings of $3$-tensors: 
    \[\T_3^k\subsetneq\Tnc^k\subsetneq\T^k.\]
\end{theorem*}

\begin{figure}[h]
    \centering
    \tikzset{every picture/.style={line width=0.75pt}} \resizebox{0.9\textwidth}{!}{
\begin{tikzpicture}[x=0.75pt,y=0.75pt,yscale=-1,xscale=1,line join=round]

\draw  [draw opacity=1, thin][fill={rgb, 255:red, 65; green, 117; blue, 5 }  ,fill opacity=1.0 ] (134,37.6) .. controls (134,27.88) and (141.88,20) .. (151.6,20) -- (447.9,20) .. controls (457.62,20) and (465.5,27.88) .. (465.5,37.6) -- (465.5,90.4) .. controls (465.5,100.12) and (457.62,108) .. (447.9,108) -- (151.6,108) .. controls (141.88,108) and (134,100.12) .. (134,90.4) -- cycle ;
\draw  [draw opacity=1, thin][fill={rgb, 255:red, 126; green, 211; blue, 33 }  ,fill opacity=1 ] (144,66.5) .. controls (144,46.89) and (205.45,31) .. (281.25,31) .. controls (357.05,31) and (418.5,46.89) .. (418.5,66.5) .. controls (418.5,86.11) and (357.05,102) .. (281.25,102) .. controls (205.45,102) and (144,86.11) .. (144,66.5) -- cycle ;
\draw  [draw opacity=1, thin][fill={rgb, 255:red, 200; green, 233; blue, 170 }  ,fill opacity=1 ] (203.5,50.4) .. controls (203.5,44.66) and (208.16,40) .. (213.9,40) -- (313.6,40) .. controls (319.34,40) and (324,44.66) .. (324,50.4) -- (324,81.6) .. controls (324,87.34) and (319.34,92) .. (313.6,92) -- (213.9,92) .. controls (208.16,92) and (203.5,87.34) .. (203.5,81.6) -- cycle ;
\draw   (219.75,47) -- (228.5,60) -- (211,60)--cycle;
\draw  [fill=gray, fill opacity=0.5] (272.75,70) -- (281.5,83) -- (264,83)--cycle;
\draw    (272.75,70) -- (289.5,74) -- (264,83)--cycle;
\draw    (272.75,70) -- (289.5,74) -- (281.5,83)--cycle;
\draw  [draw opacity=1, thin][fill={rgb, 255:red, 219; green, 248; blue, 28 }  ,fill opacity=0.59 ] (301.96,88.62) .. controls (299.41,80.61) and (307.21,70.97) .. (319.38,67.09) .. controls (331.55,63.22) and (343.49,66.57) .. (346.04,74.58) .. controls (348.59,82.59) and (340.79,92.23) .. (328.62,96.11) .. controls (316.45,99.98) and (304.51,96.63) .. (301.96,88.62) -- cycle ;

\draw[white] (442,27) node [anchor=north west][inner sep=0.75pt]  []  {$\mathcal{T}$};
\draw (360,42) node [anchor=north west][inner sep=0.75pt]  []  {$\mathcal{T}_{\mathrm{lq}}$};
\draw[white] (428,74.4) node [anchor=north west][inner sep=0.75pt]  []  {$\mathsf{S}_{p}$};
\draw (298,44.4) node [anchor=north west][inner sep=0.75pt]  []  {$\mathcal{T}_{3}$};
\draw (231,47.4) node [anchor=north west][inner sep=0.75pt]  [font=\scriptsize]  {$=\langle 2,2,2\rangle $};
\draw (218,72.4) node [anchor=north west][inner sep=0.75pt]  [font=\scriptsize]  {$\mathrm{Dome}=$};
\draw (313,68.4) node [anchor=north west][inner sep=0.75pt]  [font=\scriptsize]  {$\mathsf{AME}$};
\draw (327,80) node [anchor=north west][inner sep=0.75pt]  [font=\small]  {$\mathsf{L}$};
\draw (355,65) node [anchor=north west][inner sep=0.75pt] [font=\small] {$\det_k$};
\draw (385,65) node [anchor=north west][inner sep=0.75pt]  [font=\small]  {$\mathsf{W}_{4}$};

\end{tikzpicture}
}
\caption{Venn diagram for the semirings $\T_3 \subsetneq \Tnc \subsetneq \T$ of $k$-tensors (when $k \geq 4$) and the set of Absolutely Maximally Entangled (AME) tensors. 
    That $\Tnc \subsetneq \T$ is witnessed by an explicit tensor $\mathsf{S}_p$ (\cref{prop:psi p gives separation}). 
    Similarly, $\W_4 \in \Tnc^4 \setminus \T_3^4$ by \cref{thm:W4-non-crossing-agree}. 
    For the definition of the tensor $\mathsf{L}$, see~\cref{eq:AME-state}. For the inclusion of~$\det_k = e_1 \wedge \dotsb \wedge e_k \in (\C^k)^{\otimes k}$ and $\mathsf{AME}$ in $\Tnc$, see \cref{eq:AME-state} and \cref{remark:AME-states-are-in-Tnc}. 
    For the inclusion of the semiring generated by embeddings of $3$-tensors, $\T_3 \subset \Tnc$, see \cref{lem:embedded-tensors}.}
    \label{fig:venn-semirings}
\end{figure}

\paragraph{Existence of unknown universal spectral points.}

We show that any laminarly-supported distribution $\theta\in\distsnc$ can be identified by evaluating its quantum functionals on tensors in the laminar-quantum-robust sub-semiring $\Tnc^k$: the map $\theta \mapsto (F^\theta\big( \unitk{2}{S} \big))_{\emptyset \neq S \subsetneq [k]}$ is injective (\cref{lem:recognition}), where~$\unitk{2}{S}$ is the rank-$2$ unit tensor on~$S \subseteq [k]$. This means every laminarly-supported distribution that is not singleton-supported has a unique quantum functional that differs from the rest on the laminar-quantum-robust sub-semiring.
In other words, they are new functions in the asymptotic spectrum of $\Tnc$. New functions do not necessarily bring about new asymptotic information. We show that these do -- they can rule out certain asymptotic restrictions that can not be ruled out by any previously known spectral points (\cref{thm:laminar functionals give new information}): 
\begin{theorem*}
  For~$k \geq 4$, there exist tensors~$\psi, \varphi \in \Tnc^k$ such that~$F(\psi) \geq F(\varphi)$ for all singleton(-after-grouping) quantum functionals, and there exists some~$\theta \in \distsnc$ such that~$F^\theta(\psi) < F^\theta(\varphi)$.
\end{theorem*}

One of the fundamental mechanics of Strassen's asymptotic spectrum theory is the existence of a universal spectral point extending any non-universal spectral point for any sub-semiring of the tensors.
We know that the lower quantum functionals associated with distributions that are laminarly-supported but not singleton-supported are not universal spectral points (by \cref{prop:general-lower-func-not-spectral}). But there is \emph{some} extension of them to a universal spectral point, due to a general extension result of Strassen~\cite{Strassen1988spectrum,wigdersonAsymptoticSpectraTheorya} (see \cref{prop:Ftheta admits extension,lem:recognition}):
\begin{theorem*}
    For every~$\theta \in \distsnc$ there exists~$F \in \asympspec(\T^k)$, such that~$F(\psi) = F^\theta(\psi) = F_\theta(\psi)$ for all~$\psi \in \Tnc^k$.
    Moreover, if~$\theta \in \distsnc \setminus \distssing$ then $ F \neq F_\theta$.
\end{theorem*}

We further note that if~$\theta\in\distsnc$ is not singleton-supported, i.e. $\theta\notin\distssing$, and not singleton-supported-after-grouping, then the extension $F$ is a universal spectral point that is not currently known.

Omitting the cases of singleton-supported-after-grouping relies on the observation that for any $\ell<k$, treating a $k$-tensor as an $\ell$-tensor by grouping of local vector spaces together (i.e. by a partition of $[k]$ into $\ell$ sets) induces universal spectral points in $\asympspec(\T^k)$ from the universal spectral point $\asympspec(\T^\ell)$. 
The flattening rank, for example, is a singleton-after-grouping quantum functional $\trank_S=F^{\theta}$ for $\theta(S)=1$. 

There is no known way to interpolate between general spectral points (\cite[\S 5]{maxim2023masters} is a discussion including candidates), and so the existence of such an extension~$F$ which behaves like~$F^\theta = F_\theta$ on~$\Tnc$ is new.
An open question is whether one can establish the existence of an extension~$F$ such that~$F_\theta \leq F \leq F^\theta$ on all tensors, for instance using the techniques from~\cite{bugarInterpolatingRenyiEntanglement2025}.

\subsection{Related work}
Independent and concurrent work of Alman, Li and Pratt~\cite{almanEdgeAsymptoticSpectrum2026} also established that spectral points beyond the quantum functionals exist for higher-order tensors, as a byproduct of a more general criterion for when quantities can be characterized as a minimum of a family of universal spectral points.
The above result that the laminar quantum functionals obstruct asymptotic restrictions, in situations where the singleton quantum functionals do not, is inspired by their results (but does not rely on it).

\subsection{Organization}
The paper is organized as follows.
In \cref{sec:definitions} we collect some preliminaries on asymptotic spectra, tensor embeddings, and the upper- and lower quantum functionals.
In \cref{sec:separation} we provide a separation between the upper- and lower quantum functionals for any laminarly-supported distribution that is not singleton-supported, and show the corresponding lower quantum functionals are not universal spectral points.
We develop the new determinant-based bound here, with some details deferred to \cref{sec:max_ent_with_det,sec:werner-gamma}.
In \cref{sec:new-functionals} we define and investigate the laminar-quantum-robust sub-semiring $\Tnc^k$, showing that it contains both embedded  lower-order laminar-quantum-robust tensors, as well as the tensor~$\W_4$.
In \cref{sec:char_forms} we recall the character formula for isotypic projectors, which we use in \cref{sec:crossing} to show that upper quantum functionals can be non-normalized for distributions with non-laminar supports.

\section{Preliminaries}\label{sec:definitions}
\subsection{Tensors, restriction and the asymptotic spectrum}
We first recall the key notions of Strassen's theory, based on the exposition by Wigderson and Zuiddam in~\cite{wigdersonAsymptoticSpectraTheorya}.
We specialize our discussion to tensors, and in particular to the following object:
\begin{definition}
    Let~$\T^k$ be the set of all non-zero complex $k$-tensors, 
    \begin{equation}
        \T^k \coloneq \bigcup_{d_1,\dots,d_k\in\N}(\C^{d_1} \otimes \dotsb \otimes \C^{d_k} )\setminus \{0\}.
    \end{equation}
\end{definition}

When $\T^{k}$ is equipped with the direct sum~$\oplus$ as addition, and the Kronecker product~$\kron$ as multiplication, the set~$\T^k$ forms a commutative semiring: both operations are associative and commutative, multiplication distributes over addition, and there is a multiplicative unit $k$-tensor~$\unitk{1}{k} \coloneqq e_1 \ot \dotsb \ot e_1 \in \C^1 \ot \dotsb \ot \C^1$.
The rank-$n$ unit $k$-tensor is then
\begin{equation}
    \label{eq:rank-n-unit-k-tensor}
    \unitk{n}{k} = \sum_{i=1}^n e_i \ot \dotsb \ot e_i.
\end{equation}
Note the rank $n$ unit tensor, $\unitk{n}{k}$, can also be viewed as the direct sum~$\unitk{n}{k} = \underbrace{\unitk{1}{k} \oplus \cdots \oplus\unitk{1}{k}}_{n}$.

In addition, the set $\T^{k}$ can also be equipped with various preorder relations.
\begin{definition}
    For two tensors~$\psi, \varphi \in \T^k$, say in~$\C^{d_1} \ot \dotsb \ot \C^{d_k}$ and~$\C^{d_1'} \ot \dotsb \ot \C^{d_k'}$ respectively, we say that~$\psi$ \emph{restricts to} $\varphi$, denoted~$\varphi \leq \psi$, if and only if there exist matrices~$M_j \in \C^{d_j' \times d_j}$, $j \in [k]$, such that $\varphi = (M_1 \ot \dotsb \ot M_k) \psi$.
    We say that~$\psi$ \emph{degenerates to} $\varphi$, denoted~$\varphi \degenleq \psi$, if and only if there exist sequences of matrices~$M_j(m) \in \C^{d_j' \times d_j}$, $j \in [k]$, such that $\varphi = \lim_{m \to \infty} (M_1(m) \ot \dotsb \ot M_k(m)) \psi$.
\end{definition}
The restriction relation, $\leq$, and degeneration relation, $\degenleq$, are both reflexive and transitive. 
They are also both compatible with the direct sum $\oplus$ and Kronecker product $\kron$,
meaning the set of $k$-tensors, $\T^k$, forms a preordered semiring.

Note the degeneration preorder may also be formulated more geometrically (through orbit closures) or algebraically (where the~$M_j$ are matrices with entries rational functions in a parameter~$\eps$), see e.g.~\cite{Strassen1987relative} or~\cite[Thm.~20.24]{burgisserAlgebraicComplexityTheory1997}.

The restriction preorder is additionally a \emph{Strassen preorder} as defined in~\cite{wigdersonAsymptoticSpectraTheorya}. 
Although the precise definition of a Strassen preorder is not relevant for this paper, an important property is that it is inherited by sub-semirings with the restricted order.
\begin{definition}
    Let~$\semiring$ be a semiring with a Strassen preorder~$\leq$.
    The \emph{closure}~$\asympleq$ of~$\leq$ is defined as follows: for~$a, b \in \semiring$, we say~$a \asympleq b$ if there exists~$f\colon \N \to \N$ such that~$f(n) = o(n)$ and~$a^n \leq b^n 2^{f(n)}$.
\end{definition}
\begin{definition}
    Let~$\semiring$ be a semiring with a Strassen preorder~$\leq$.
    The \emph{asymptotic spectrum}~$\asympspec(\semiring, \leq)$ is the set of all~$\leq$-monotone homomorphisms~$F\colon \semiring \to \R_{\geq 1}$.
    Its elements are called \emph{spectral points}.
    We write~$\asympspec(\semiring)$ if~$\leq$ is clear from context.
\end{definition}
One of the essential features of Strassen's theory is the following duality theorem:
\begin{theorem}[Duality, {\cite[Thm.~3.22]{wigdersonAsymptoticSpectraTheorya}}]
    Let~$\leq$ be a Strassen preorder on~$\semiring$.
    Let~$a,b \in \semiring$.
    We have~$a \asympleq b$ if and only if for every~$F \in \asympspec(\semiring, \leq)$ it holds that~$F(a) \leq F(b)$.
\end{theorem}
The asymptotic spectrum also possesses the property that, for inclusions~$\semiring \subseteq \semiring'$ with the order given by the restriction of the order in~$\semiring'$, the spectral points on~$\semiring$ admit extensions to~$\semiring'$:
\begin{theorem}[{\cite[Prop.~3.3]{Strassen1988spectrum}, see also~\cite[Thm.~4.16]{wigdersonAsymptoticSpectraTheorya}}]
  \label{thm:spectral point extension}
    Let~$\semiring, \semiring'$ be semirings with Strassen preorders~$\leq_\semiring, \leq_{\semiring'}$ respectively.
    Suppose~$f\colon \semiring \to \semiring'$ is a semiring homomorphism that is also an order embedding, i.e., $a \leq_\semiring b$ if and only if~$f(a) \leq_{\semiring'} f(b)$.
    Then the map~$f^* \colon \asympspec(\semiring', \leq_{\semiring'}) \to \asympspec(\semiring, \leq_\semiring)$, given by~$F \mapsto F \circ f$, is surjective.
\end{theorem}

\subsection{Embedded tensors}
When discussing higher-order tensors, it is helpful to note that every $\ell$-tensor $\psi$ can be viewed as a $k$-tensor for $k > \ell$ by tensoring onto $\psi$ unit tensors in a way which respects the aforementioend preorder relations.
We refer to $k$-tensors obtained in this way as \textit{embedded} tensors and formalize the definition as follows.
\begin{definition}
  \label{def:embedded tensor}
  Let~$\ell < k$ be positive integers and let $\psi \in \C^{d_1} \ot \dotsb \ot \C^{d_\ell}$ be an $\ell$ tensor.
  The standard embedding $\psi$ as a $k$-tensor is defined as
  \[
    \tilde\psi \coloneqq \psi \ot (\bigotimes_{j =\ell+1}^k e_1) \in \C^{d_1} \ot \dotsb \ot \C^{d_{\ell}} \ot (\C^1)^{\ot (k-\ell)}.
  \]
  More generally, let $S = (s_1, \ldots, s_\ell) \subseteq [k]$ be a given \textit{ordered} subset of $[k]$ of size $\ell$.
  Then the $S$-embedding of $\psi$ is the $k$-tensor obtained by transposing $\widetilde\psi$ so that the first $\ell$ legs of the standard embedding $\widetilde \psi$ are mapped to the positions $(s_1,\ldots, s_\ell)$ in $[k]$.
\end{definition}

\subsection{Upper and lower quantum functionals}
\label{sec:quantum_functionals}
We now recall the upper and quantum functionals as defined in~\cite{cvz23universal}. They are defined with respect to weightings on bipartitions of~$[k]$.
\begin{definition} \label{def:mutually-laminar-bipartitions}
  A bipartition of $[k]$ is an unordered pair~$\{S, \overline S\}$ with~$S \subseteq [k]$, $S \neq \emptyset$ and~$[k] \setminus S = \overline S \neq \emptyset$.
  Two bipartitions~$\{S, \overline S\}$ and~$\{T, \overline T\}$ are \emph{mutually laminar} if we can choose sides that contain each other, i.e. if there are $A\in\{S,\overline S\},B\in\{T,\overline T\}$ such that  $A\subseteq B$ (this is symmetric, as $A\subseteq B\iff\overline{B}\subseteq\overline{A}$ ).
  The set of bipartitions of~$[k]$ is denoted by~$\bipartitions^k$, or just $\bipartitions$ if $k$ is clear from the context.
\end{definition}

We often use just $S$ to denote $\{S,\overline{S}\}$. We also use letters to denote the parties in $S$ instead of numerals when $k$ is small, e.g. $AB$ is the bipartition $\{\{1,2\},[k]\setminus\{1,2\}\}$. This is sometimes denoted by $AB|CD$ (for $k=4$).

\begin{definition} \label{def:laminar-bipartitions-set}
  A set of bipartitions of $[k]$ is called \emph{laminar} if every pair of bipartitions in it is mutually laminar (as defined by~\cref{def:mutually-laminar-bipartitions}).
\end{definition}

The set of probability distributions on bipartitions of~$[k]$ is denoted~$\dists^k$, or simply~$\dists$ if~$k$ is clear from context.
A distribution~$\theta \in \dists^k$ is said to be laminarly-supported if its support is a laminar set of bipartitions, and the set of laminarly-supported distributions is denoted $\distsnc^k$ or $\distsnc$. The set of distributions which are supported only on singleton bipartitions $\{\{i\},[k]\setminus\{i\}\}$ is denoted $\distssing^k$ or $\distssing$.

A laminar set of bipartitions was originally called non-crossing in~\cite{cvz23universal}. It is closely connected to the combinatorics notion of a laminar family of sets: a family in which any two sets either have a containment relation or are disjoint. A set of bipartitions is laminar iff there are consistent choices of sides for each bipartition, such that the chosen sides constitute a laminar family of sets. To see that such a consistent choice exists, fix a bipartition~$S_0 | \overline{S_0}$ with~$\abs{S_0}$ maximal, and then for every bipartition~$S|\overline{S}$ select~$S$ if~$S \subseteq S_0$, otherwise choose $\overline{S}$.

For the definition of the upper quantum functionals, Schur--Weyl duality is essential.
Consider a finite-dimensional complex vector space~$V$.
Then one can consider the actions of~$\GL(V)$ and the symmetric group~$S_n$ on~$V^{\otimes n}$, where~$\GL(V)$ acts via~$g \cdot (v_1 \ot \dotsb \ot v_n) = gv_1 \ot \dotsb \ot gv_n$ for~$g \in \GL(V)$, and~$S_n$ acts via~$\pi \cdot v_1 \ot \dotsb \ot v_n = v_{\pi^{-1}(1)} \ot \dotsb \ot v_{\pi^{-1}(n)}$ for~$\pi \in S_n$.
These actions commute, and in fact the spans of the maps~$\GL(V) \to \End(V^{\ot n})$ and~$S_n \to \End(V^{\ot n})$ are each others commutants.
This implies that~$V^{\ot n}$ admits a decomposition into irreducibles of the form
\[
  V^{\ot n} \cong \bigoplus_{\lambda \vdash n} \mathbb S_\lambda(V) \ot [\lambda]
\]
where the~$\mathbb S_\lambda(V)$ are the irreducible~$\GL(V)$-representations labeled by~$\lambda$, and~$[\lambda]$ are the irreducible~$S_n$-representations labeled by~$\lambda$.

When~$V = V_1 \ot \dotsb \ot V_k$, one can choose to apply Schur--Weyl duality merely on a subset of the factors.
For $b = \{S, \overline{S}\} \in \bipartitions^k$ and $\lambda \vdash n$, we define $P_\lambda^{V_b} = P_\lambda^{V_S} P_{(n)}^{V_{[k]}}$.
Here $P_\lambda^{V_S}$ is the projector onto the $\lambda$-term in
\[
  (V_1 \ot \dotsb \ot V_k)^{\ot n} \cong V_S^{\ot n} \ot V_{\overline{S}}^{\ot n} \cong \bigoplus_{\lambda \vdash n} \mathbb S_\lambda(V_S) \ot [\lambda] \ot (V_{\overline S})^{\ot n}.
\]

For mutually laminar bipartitions, such projectors commute:
\begin{lemma}[{\cite[Lem.~3.2]{cvz23universal}}]
  \label{lem:nc projectors commute}
  Suppose~$b = \{S, \overline{S}\}$ and~$b' = \{T, \overline{T}\}$ are mutually laminar bipartitions of~$[k]$.
  Then for any~$\lambda, \mu \vdash n$, $P_\lambda^{V_b}$ and~$P_\mu^{V_{b'}}$ commute.
\end{lemma}
\begin{definition}[Entropy]
  For a probability distribution~$p_1, \dotsc, p_m$ on~$m$ outcomes, we define the Shannon entropy~$H(p) = - \sum_{i=1}^m p_i \log_2(p_i)$.
  Similarly, for a positive definite operator~$\rho\colon \mathcal{H} \to \mathcal{H}$ on a Hilbert space~$\mathcal{H}$, with trace~$1$, we define its von Neumann entropy~$H(\rho)$ as the Shannon entropy of its eigenvalues.
\end{definition}
\begin{definition}[Upper quantum functional]
  \label{def:upper qfunc}
  Let~$\psi \in V_1 \ot \dotsb \ot V_k$ and~$\theta \in \distsnc$. Define
  \[
    E^\theta(\psi) = \sup_{n \geq 1} \sup_{(\lambda^{(b)})} \sum_{b \in \supp \theta} \theta(b) H(\overline{\lambda^{(b)}}), \quad \overline{\lambda^{(b)}} = \frac{\lambda^{(b)}}{n},
  \]
  with the supremum over all tuples $(\lambda^{(b)})_{b \in \supp \theta}$ of Young diagrams $\lambda^{(b)} \vdash n$ such that
  \[
    \left(\prod_{b \in \supp \theta} P_{\lambda^{(b)}}^{V_b}\right) \psi^{\kron n} \neq 0.
  \]
  The product of the projectors is independent of the order, by virtue of~\cref{lem:nc projectors commute}.
  The upper quantum functional~$F^\theta$ is defined by~$F^\theta(t) = 2^{E^\theta}$.
  Note that~$F^\theta(0) = 0$.
\end{definition}

\begin{proposition}[{\cite[Thm.~3.5]{cvz23universal}}]
    \label{prop:upper qfunc properties}
    For~$\theta \in \distsnc^k$, the upper quantum functional~$F^\theta$ satisfies:
    \begin{enumerate}
        \item $F^\theta(\unitk{n}{k}) = n$,
        \item $F^\theta(\psi \oplus \varphi) \leq F^\theta(\psi) + F^\theta(\varphi)$,
        \item $F^\theta(\psi \kron \varphi) \leq F^\theta(\psi) F^\theta(\varphi)$,
        \item $F^\theta(\psi) \geq F^\theta(\varphi)$ when~$\psi \degengeq \varphi$.
    \end{enumerate}
\end{proposition}

Suppose~$0 \neq \psi \in V_1 \ot \dotsb \ot V_k$ and~$S \subsetneq [k]$ is non-empty.
The~$S$-flattening of~$\psi$, denoted by~$\psi_S$, is~$\psi$ reinterpreted as a linear map~$V_{\overline{S}}^* \to V_{S}$.
\begin{definition}
    For a bipartition $b=\{S,\overline{S}\}$ ($ S,\overline{S}\neq\emptyset$), its flattening rank is the function $\trank_{b}(\psi)=\rank(\psi_S)$.
    We may also denote it by~$\trank_S(\psi)$.
\end{definition}
Note that~$\rank(\psi_S) = \rank(\psi_{\overline{S}})$, so~$R_b$ depends only on the bipartition.

Suppose now that each of the~$V_j$ is endowed with some inner product.
Then the adjoint~$\psi_S^*\colon V_{S}^* \to V_{\overline{S}}$ can be identified with a linear map~$V_{S}\to V_{\overline{S}}^*$, and so we may consider the composition
\[
  \rho_S \coloneq \frac{\psi_S \psi_S^*}{\norm{\psi}_2^2}
\]
This is a positive semidefinite operator on~$V_S$ with trace~$1$, and we will refer to it as the~$S$-marginal of~$\psi$.
For the bipartition~$b = \{S, \overline{S}\}$, we will write
\[
  H_b(\psi) = H_S(\psi) \coloneq H(\rho_S)
\]
for the marginal entropy of~$\psi$ across the bipartition~$b$.
This depends only on~$b$ and not on~$S$, as one can show that~$\rho_S$ and~$\rho_{\overline{S}}$ are isospectral.
If~$\psi = 0$ we define~$H_b(\psi) = -\infty$.

For a probability distribution $\theta\in \dists$, we let $H_\theta(\psi)$ be the weighted average of the marginal entropies of the normalization of the tensor $\psi$:
\[
    H_\theta(\psi) = \sum_{b\in B}\theta(b)H_b(\psi)
\]

\begin{definition}[Lower quantum functional]
  \label{def:lower qfunc}
  Let~$0 \neq \psi \in V_1 \ot \dotsb \ot V_k$ and~$\theta \in \dists$. Define
  \[
    E_\theta(\psi) = \sup_{\substack{A_1, \dotsc, A_k \\ A_j \in \GL(V_j)}} H_\theta ((A_1 \ot \dotsb \ot A_k) \psi)
  \]
  Then the lower quantum functional~$F_\theta$ is defined by~$F_\theta(\psi) = 2^{E_\theta}$.
  We additionally define~$F_\theta(0) = 0$.
\end{definition}
\begin{remark}
    One connection between the above and quantum information theory is the following.
    If one interprets~$\psi$ as a multipartite quantum state and uses Dirac notation instead, then~$\rho = \frac{\ket{\psi}\bra{\psi}}{\braket{\psi|\psi}} = \frac{\psi \psi^*}{\norm{\psi}^2}$ is the associated density matrix.
    The object~$\rho_S$ is precisely the \emph{reduced density matrices} of~$\rho$ on~$S$, i.e., $\rho_S$ is the \emph{partial trace} of~$\rho$ over~$\overline{S}$.
\end{remark}

We later require one additional lemma, providing a slightly different characterization of the lower quantum functional:
\begin{lemma}
  \label{lem:lower qfunc proj orbit closure}
  Let~$V = V_1 \ot \dotsb \ot V_k$, and let~$[\varphi] \in \Proj(V)$ denote the projective equivalence class of~$\varphi \in V \setminus \{0\}$.
  We have that
  \begin{align*}
    E_\theta(\psi) = \max_{[\varphi] \in \overline{\SL \cdot [\psi]}} H_{\theta}(\varphi) = \max_{\varphi \in \overline{\GL \cdot \psi}} H_\theta(\varphi).
  \end{align*}
\end{lemma}
\begin{proof}
  The equality of the two maxima is clear, as~$H_\theta$ is invariant under scalar multiplication.
  Clearly we have
  \[
    E_\theta(\psi) = \sup_{[\varphi] \in \SL \cdot [\psi]} H_\theta(\psi) \leq \max_{[\varphi] \in \overline{\SL \cdot [\psi]}} H_{\theta}(\varphi)
  \]
  The supremum equals the maximum over the orbit closure, as $H_\theta$ is a continuous function on~$\Proj(V)$.
\end{proof}

\begin{proposition}[{\cite[Thm.~3.19]{cvz23universal}}]
    \label{prop:lower qfunc properties}
    For~$\theta \in \dists$, the lower quantum functional~$F_\theta$ satisfies:
    \begin{enumerate}
        \item $F_\theta(\unitk{n}{k}) = n$,
        \item $F_\theta(\psi \oplus \varphi) \geq F_\theta(\psi) + F_\theta(\varphi)$,
        \item $F_\theta(\psi \kron \varphi) \geq F_\theta(\psi) F_\theta(\varphi)$,
        \item $F_\theta(\psi) \geq F_\theta(\varphi)$ when~$\psi \degengeq \varphi$.
    \end{enumerate}
\end{proposition}

The upper and lower quantum functionals agree for singleton-supported distributions, hence form universal spectral points:
\begin{theorem}[{\cite[Thm.~3.30]{cvz23universal}}]
  \label{thm:sing-qfunc-agree}
  Let~$\theta \in \distssing$. Then~$E_\theta(\psi) = E^\theta(\psi)$.
\end{theorem}
\begin{corollary}[{\cite[Cor.~3.31]{cvz23universal}}]
  \label{cor:sing-qfunc-spectral}
  For~$\theta \in \distssing$, $F^\theta = F_\theta$ is a point in the asymptotic spectrum~$\asympspec(\T^k)$ of~$k$-tensors.
\end{corollary}

\section{Separating upper and lower quantum functionals}\label{sec:separation}

In this section we establish separations between the upper and lower quantum functionals.
In~\cref{subsec:qualitative-separation} we establish that such separations exist for all~$k \geq 4$ and $\theta \in \distsnc^k \setminus \distssing^k$, using the Kempf--Ness theorem~\cite{kempf1979stability,nessStratificationNullCone1984}.
In~\cref{subsec:quantitative-separation} we strengthen these results to quantitative separations, by proving upper bounds on the lower quantum functional.
This strengthening allows us to show that separation between the upper and lower quantum functionals occurs generically, and that the lower quantum functionals are not spectral points.

\subsection{Qualitative separation}
\label{subsec:qualitative-separation}

We first establish the fact that separations exist.
\begin{theorem}
    \label{thm:upper lower separation general k}
    Let $k \geq 4$ and let $\theta \in \distsnc^k \setminus \distssing^k$ be a distribution on bipartitions of $[k]$ such that $\theta$ is laminar and also non-singleton-supported.
    Then there exists a $k$-tensor $\psi$ such that
    \begin{equation}
        E_{\theta}(\psi)<E^{\theta}(\psi).
    \end{equation}
\end{theorem}

We begin with a simple upper bound on upper quantum functionals of laminarly-supported distributions.
\begin{proposition}
    \label{prop:E_upper_M_bound}
    Let $\theta \in \distsnc^k$ be a laminarly-supported distribution\footnotemark{} on bipartitions of $[k]$.
    Then for any $k$-tensor $\psi$,
    \begin{equation*}
        E^{\theta}(\psi) \leq M^{\theta}(\psi)
    \end{equation*}
    where $M^{\theta}(\psi)$ is defined as the $\theta$-sum of the logarithms of the flattening ranks of $\psi$:
    \begin{equation*}
        M^{\theta}(\psi) = \sum_{b \in \supp \theta}\theta(b) \log_2 \trank_{b}(\psi).
    \end{equation*}
\end{proposition}
\vspace{-1em}
\footnotetext{
    In~\cref{sec:crossing}, we provide an example of a distribution that has a non-laminar support $\theta \not\in \distsnc$ on $(AB|CD)$ and $(BC|AD)$ for which the value of $E^{\theta}(\psi)$ and the validity of the upper bound $E^{\theta}(\psi) \leq M^{\theta}(\psi)$ depends on the ordering of the projection operators, $P^{AB|CD}_{\mu}$ and $P^{BC|AD}_{\nu}$, used for defining $E^{\theta}(\psi)$. Hence, the assumption that $\theta \in \distsnc$ is essential.
}
\begin{proof}
    As $\theta$ is a distribution with laminar support, we know that the corresponding set of projection operators $\{P^{V_b}_{\lambda_b} | b \in \supp(\theta) \}$ is completely mutually commuting. As a consequence, $\prod_{b \in \supp \theta} P_{\lambda^{(b)}}^{V_b} \psi^{\ot n} \neq 0$ implies that $P^{V_b}_{\lambda_b} \psi^{\otimes n} \neq 0$ for each $b \in \supp \theta$.
    The claim then holds because the $b$-th flattening rank $\trank_{b}(\psi)$ is an upper bound on the number of rows of a Young diagram $\lambda_b \vdash k$ such that $P^{V_b}_{\lambda_b} \psi^{\otimes k} \neq 0$ and moreover the maximum entropy of all such Young diagrams is $\log_2 \trank_{b}(\psi)$ (i.e., the entropy of a rectangular Young diagram with $\trank_{b}(\psi)$ rows).
\end{proof}
The same upper bound holds for the lower quantum functional, since~$E_\theta \leq E^\theta$.\footnote{For general~$\theta \in \dists$, the bound~$E_\theta \leq M^\theta$ can also be established by observing~$H_b(\varphi) \leq \log_2 R_b(\varphi) \leq \log_2 R_b(\psi)$ for any~$\varphi$ in the~$\SL_{d_1} \times \dotsb \times \SL_{d_k}$-orbit closure of~$\psi$.}
Moreover, one trivially has the bound~$H_\theta(\psi) \leq E_\theta(\psi)$.
Consequently, the following implication holds:
\begin{equation}
    \label{eq:entropy_rank_implication}
    \text{if} \quad H_{\theta}(\psi) = M^{\theta}(\psi), \quad \text{then} \quad E_{\theta}(\psi) = M^{\theta}(\psi).
\end{equation}
Of course, the converse statement does not hold in general, as there are examples of tensors for which $H_{\theta}(\psi) < E_{\theta}(\psi) = M^{\theta}(\psi)$.
Nevertheless, for tensors with maximal singleton entropies and maximal non-singleton flattening ranks, we will establish the converse to \cref{eq:entropy_rank_implication}.
The primary tool is a classical result of Kempf and Ness:\footnote{The form of the theorem as we state it, is slightly different from the original works.
The original statements involve only the orbit, rather than the orbit closure.
However, every orbit closure only contains a single closed orbit, as by a theorem of Mumford~\cite[Cor.~1.2]{mumfordGeometricInvariantTheory1994}, any two closed orbits can be separated by a polynomial invariant, which is necessarily constant on orbit closures.
Moreover, vanishing of the moment map~$\mu$ on~$x_0$ implies that the orbit of~$x_0$ is closed.}
\begin{theorem}[{\cite{kempf1979stability,nessStratificationNullCone1984}}]
  \label{thm:kempf-ness-general}
  Consider a rational action of a connected complex reductive group~$G$ on a complex vector space~$V$.
  Let~$K$ be a maximal compact subgroup of~$G$ and assume~$V$ is endowed with a~$K$-invariant inner product.
  Let~$\mu\colon \Proj(V) \to \Lie(K)^*$ be the moment map for the action.
  Then for semistable~$x \in \Proj(V)$, there exists~$x_0 \in \overline{G \cdot x}$ such that
  \[
    \mu^{-1}(0) \cap \overline{G \cdot x} = K \cdot x_0.
  \]
  Moreover, if~$x = [v]$, then there exists~$v_0 \in \overline{G \cdot x} \subseteq V$ such that~$x_0 = [v_0]$, and~$\norm{g \cdot v} \geq \norm{v_0}$ for all~$g \in G$.
\end{theorem}
More succinctly, in $\overline{G \cdot v}$ there is at most one~$K$-orbit on which~$\mu$ evaluates to~$0$, and these form exactly the points of minimal norm inside the orbit closure.

We describe now how to interpret this result in our concrete setting. The group of interest is~$G = \SL_{d_1} \times \dotsb \times \SL_{d_k}$, acting on~$\Proj(V)$ with~$V = \C^{d_1} \otimes \dotsb \otimes \C^{d_k}$.
The dual of the Lie algebra~$\Lie(K)^*$ can be identified with tuples~$(H_1, \dotsc, H_k)$, with each~$H_j$ a $d_j \times d_j$ trace-zero Hermitian matrix.
The moment map for the action takes the following explicit form:
\begin{proposition}
  \label{prop:moment map}
  The moment map~$\mu\colon \Proj(V) \to \Lie(K)^*$ is given by
  \[
    \mu([\psi]) = \frac{1}{\norm{\psi}_2^2}(\psi_1 \psi_1^*, \dotsc, \psi_k \psi_k^*) - (\frac{I_{d_1}}{d_1}, \dotsc, \frac{I_{d_k}}{d_k}).
  \]
\end{proposition}
The Kempf--Ness theorem then implies the following:
\begin{corollary}
  \label{cor:kempf-ness-tensor}
  Let $V = \C^{d_1}\ot\cdots \ot \C^{d_k}$, $K = \SU_{d_1} \times \cdots \times \SU_{d_k}$ and $G = \SL_{d_1} \times \cdots \times \SL_{d_k}$.
  If~$\psi \in V \setminus \{0\}$ and $[\varphi] \in \overline{G \cdot [\psi]} \subseteq \Proj(V)$ are such~$\mu([\psi]) = \mu([\varphi]) = 0$, i.e., for all singletons $s$,
  \[
    \rho^\psi_{s} = \rho^{\varphi}_s = \frac{I}{d_s},
  \]
  then~$[\varphi] \in K \cdot [\psi]$.
  Moreover, $\norm{g \cdot \psi}_2 \geq \norm{\psi}_2$ for all~$g \in G$, i.e., the function~$f_\psi(g) = \norm{g \cdot \psi}_2$ is minimized at~$g = I$.
\end{corollary}
We now establish a partial converse to \cref{eq:entropy_rank_implication}:
\begin{proposition}
    \label{prop:HEM_collapse}
    Let $\theta\in \distsnc$ be a distribution with laminar support,
    and $\psi$ be a $k$-tensor such that 
    \begin{enumerate}
        \item the singleton entropies of $\psi$ are maximal, i.e.,
        \begin{equation*}
            \forall s \text{ singleton} : \quad H_{s}(\psi) = \log_2 \trank_{s}(\psi),
        \end{equation*}
        \item there exists $S_1,\ldots,S_m \subseteq[k]$ such that $\bigcup_j S_j = [k]$, for all $j$, $\theta(S_j|\overline{S_j}) > 0$ and, 
        \begin{equation*}
             \trank_{S_j|\overline{S_j}}(\psi) = \prod_{s\in S_j}\trank_{s}(\psi).
        \end{equation*}
    \end{enumerate} 
    Then $E_{\theta}(\psi) = M^{\theta}(\psi)$ implies $H_{\theta}(\psi) = M^{\theta}(\psi)$.
\end{proposition}
\begin{proof}
     The assumption $E_{\theta}(\psi) = M^{\theta}(\psi)$ together with~\cref{lem:lower qfunc proj orbit closure} implies there exists~$\varphi \in \overline{\GL \cdot \psi}$ such that~$H_\theta(\varphi) = M^{\theta}(\psi)$.
     Therefore, for each $b \in \supp \theta$, we have
$\log_2 \trank_b(\psi) = H_b(\varphi) \leq \log_2 \trank_b(\varphi) \leq \log_2 \trank_b(\psi)$.
     In particular, this holds for bipartitions of the form $S_j | \overline{S_j}$, and the $S_j$-marginal of $\varphi$ is uniform.
     Assumption (ii) now implies that for any $s \in S_j$, the $s$-marginal of $\varphi$ is also uniform, and hence $H_s(\varphi) = \log_2 \trank_s(\varphi)$.
     By assumption (i), however, the eigenvalues of the $s$-marginals of $\psi$ are uniform, which, by~\cref{cor:kempf-ness-tensor} implies $\psi$ and $\varphi$ must actually lie in the same projective $\U_{d_1}\times\cdots \times \U_{d_k}$ orbit.
     This then implies $H_{\theta}(\varphi) = H_{\theta}(\psi)$.
     Since we previously established $H_{\theta}(\varphi) = M^{\theta}(\psi)$, we conclude $H_{\theta}(\psi) = M^{\theta}(\psi)$ also, which completes the claim.
\end{proof}
\begin{corollary}
    \label{cor:separations}
    Let $\psi$ be a $k$-tensor satisfying the assumptions of \cref{prop:HEM_collapse}, and $\theta\in \distsnc$ be a laminarly-supported distribution such that~$H_\theta(\psi) < M^{\theta}(\psi)$ and also $E^{\theta}(\psi) = M^{\theta}(\psi)$.
    Then
    \begin{equation*}
        E_{\theta}(\psi) < E^{\theta}(\psi).
    \end{equation*}
\end{corollary}
\begin{proof}
    By \cref{prop:HEM_collapse}, $H_\theta(\psi) \neq M^{\theta}(\psi)$ implies $E_{\theta}(\psi) < M^{\theta}(\psi)$.
    Since $M^{\theta}(\psi) = E^{\theta}(\psi)$ by assumption, the claim holds.
\end{proof}

To establish a separation between the lower and upper quantum functionals, it therefore suffices to find a $k$-tensor which satisfies the assumptions of \cref{prop:HEM_collapse} and \cref{cor:separations}.
For each $p \in [0,1]$, consider the $4$-tensor $\mathsf{S}_p \in (\mathbb C^{2})^{\ot 4}$ given by
\[
  \mathsf{S}_p =
  \sqrt{\frac{p}{2}} \parens*{e_1^{\ot 4} + e_2^{\ot 4}}
  + \sqrt{\frac{1-p}{2}} \parens*{e_1 \ot e_2 \ot e_1 \ot e_2 + e_2 \ot e_1 \ot e_2 \ot e_1}.
\]

\begin{proposition}
    \label{prop:psi p gives separation}
    Let~$\theta \in \distsnc^4$ be such that $\theta(AB|CD) > 0$.
    Then for all $p \in (0,1/2)\cup (1/2,1)$,
    \[
        E_\theta(\mathsf{S}_p) < E^{\theta}(\mathsf{S}_p).
    \]
\end{proposition}
\begin{proof}
With respect to the standard inner product on $\mathbb C^{2}$ making $\{e_1, e_2\}$ an orthonormal basis, the eigenvalues of all singleton marginals of $\mathsf{S}_p$ are all uniform, e.g.,
\[
    \mathrm{eig}_{A|BCD}(\mathsf{S}_p) = \left(\frac{1}{2},\frac{1}{2}\right),
\]
and therefore the singleton entropies of $\mathsf{S}_p$ are maximal, e.g., $H_{A|BCD}(\mathsf{S}_p) = \log_2 R_{A|BCD}(\mathsf{S}_p) = \log_2 2$.
At the same time, for the $AB|CD$ bipartition, we have non-uniform eigenvalues
\[
    \mathrm{eig}_{AB|CD}(\mathsf{S}_p) = \left(\frac{p}{2},\frac{p}{2},\frac{1-p}{2},\frac{1-p}{2}\right).
\]
For $p \in (0,1)$, we have $R_{AB|CD}(\mathsf{S}_p) = 4$ and thus $\mathsf{S}_p$ satisfies the conditions of \cref{prop:HEM_collapse}.

For $p \in (0,1)$, the singleton and $AB|CD$ flattening ranks of $\mathsf{S}_p$ are $2$ and $4$ respectively meaning
\begin{equation}
    M^{\theta}(\mathsf{S}_p) = \begin{cases}
        (1-\theta(AB|CD)) \log_2 2 + \theta(AB|CD) \log_2 4 & p \in(0,1), \\
        \log_2 2 & p = 0,1.
    \end{cases}
\end{equation}
To apply \cref{cor:separations}, we must establish that $E^{\theta}(\mathsf{S}_p) = M^{\theta}(\mathsf{S}_p)$.
To accomplish this, note that it suffices to prove $P_{\ydiagram{1,1,1,1}}^{AB|CD} P_{\ydiagram{2,2}}^{A} P_{\ydiagram{2,2}}^{B} P_{\ydiagram{2,2}}^{C} P_{\ydiagram{2,2}}^{D} \mathsf{S}_p^{\ot 4} \neq 0$.
For the tensor space $\mathbb C^{2} \ot \mathbb C^{2} \ot \mathbb C^{2} \ot \mathbb C^{2}$, note that $P_{\ydiagram{1,1,1,1}}^{AB|CD}$, given by
\[
    P_{\ydiagram{1,1,1,1}}^{AB|CD} = (P_{\ydiagram{1,1,1,1}}^{AB} \ot I_{CD}^{\ot 4}) P_{\ydiagram{4}}^{ABCD} = (I_{AB}^{\ot 4} \ot P_{\ydiagram{1,1,1,1}}^{CD}) P_{\ydiagram{4}}^{ABCD}
\]
has rank one because $\dim(\mathbb S(\mathbb C^{4})_{\ydiagram{1,1,1,1}}) = 1$ 
and $\dim([\vcenter{\hbox{\,\ydiagram{1,1,1,1}\,}}]) = 1$.
Moreover, the restriction from $\GL_4$ on $AB$ to the tensor product subgroup $\GL_2 \times \GL_2$ on $A$ and $B$ of $\mathbb S(\mathbb C^{4})_{\ydiagram{1,1,1,1}}$ is of the form
\[
    \mathrm{Res}^{\GL_4}_{\GL_2 \times \GL_2} \mathbb S_{\ydiagram{1,1,1,1}}(\C^{4}) \simeq \mathbb S_{\ydiagram{2,2}}(\C^{2}) \ot \mathbb S_{\ydiagram{2,2}}(\C^{2}).
\]
This restriction relation then implies
$P_{\ydiagram{1,1,1,1}}^{AB|CD} P_{\ydiagram{2,2}}^{A} P_{\ydiagram{2,2}}^{B} P_{\ydiagram{2,2}}^{C} P_{\ydiagram{2,2}}^{D} = P_{\ydiagram{1,1,1,1}}^{AB|CD}$.
Therefore,
\[
    \norm{P_{\ydiagram{1,1,1,1}}^{AB|CD} P_{\ydiagram{2,2}}^{A} P_{\ydiagram{2,2}}^{B} P_{\ydiagram{2,2}}^{C} P_{\ydiagram{2,2}}^{D} \mathsf{S}_p^{\ot 4}}^{2} = \norm{P_{\ydiagram{1,1,1,1}}^{AB|CD} \mathsf{S}_p^{\ot 4}}^{2} = \abs{{\det}((\mathsf{S}_p)_{AB|CD})}^2
\]
where $(\mathsf{S}_p)_{AB|CD}$ is the $AB|CD$ flattening of $\mathsf{S}_p$.
Now note that $\abs{{\det}((\mathsf{S}_p)_{AB|CD})}^2 = \frac{p^2(1-p)^2}{16}$. This does not vanish when~$p \in (0,1)$, establishing $E^{\theta}(\mathsf{S}_p) \geq M^{\theta}(\mathsf{S}_p)$. 
By \cref{prop:E_upper_M_bound} ($\theta$ is laminar), we have $E^{\theta}(\mathsf{S}_p) \leq M^{\theta}(\mathsf{S}_p)$ also and therefore also
\[
    E^{\theta}(\mathsf{S}_p) = M^{\theta}(\mathsf{S}_p).
\]
Regarding the lower quantum functional, the $AB|CD$ entropy of $\mathsf{S}_p$ is
\begin{equation}
    H_{AB|CD}(\mathsf{S}_p) = 2h(p) = -2 p\log_2 p - 2 (1-p)\log_2 p,
\end{equation}
and therefore
\begin{equation}
    H_{\theta}(\mathsf{S}_p) = (1-\theta(AB|CD) \log_2 2 + \theta(AB|CD) \cdot 2h(p),
\end{equation}
Therefore, for all $p \in (0,1/2) \cup (1/2, 1)$, we have the strict inequality $H_{\theta}(\mathsf{S}_p) < M^{\theta}(\mathsf{S}_p)$ and thus \cref{prop:HEM_collapse} implies the strict inequality $E_{\theta}(\mathsf{S}_p) < M^{\theta}(\mathsf{S}_p) = E^{\theta}(\mathsf{S}_p)$ also.
\end{proof}
By embedding $\mathsf{S}_p$ as a higher-order tensor, we establish separations for arbitrary~$k \geq 4$ and~$\theta \in \distsnc \setminus \distssing$.
\begin{proof}[Proof of~\cref{thm:upper lower separation general k}.]
    As $\theta$ is a non-singleton-supported distribution, there exists a bipartition $b=(S,\overline{S})$ in the support of $\theta$ such that $|S|\geq 2$, and $|\overline{S}|\geq 2$.
    Choose $A,B \in S$ and $C, D \in \overline{S}$ and define $\varphi$ to be a tensor whose local dimensions are $d_A =d_B =d_C =d_D = 2$ and $d_X = 1$ for all other $X \in [k]\setminus\{A,B,C,D\}$.
    
    Let $\varphi$ be the tensor of the form
    \begin{equation}
        \varphi_{ABCDX_1\cdots X_{k-4}} = \psi_{\frac{1}{3};ABCD} \ot v_1\ot\cdots \ot v_{k-4}
    \end{equation}
    where $v_i$ are any unit vectors in the one-dimensional spaces.
    Now as $\theta$ is laminar, we know that 
    \begin{equation}
        \theta(ABX|CD X') > 0, \quad
        \theta(ACX|BD X') = 0, \quad 
        \theta(ADX|BC X') = 0,
    \end{equation}
    where $X$ denotes any string of indices in $[k]\setminus\{A,B,C,D\}$ and $X'$ its complement in $[k]\setminus\{A,B,C,D\}$.
    We already established that $E_{\tilde\theta}(\psi_{1/3}) < E^{\tilde\theta}(\psi_{1/3})$ for all distributions $\tilde \theta$ on bipartitions of the seven bipartitions of $(A,B,C,D)$ satisfying
    \begin{equation}
        \tilde \theta(AB|CD) > 0, \quad
        \tilde \theta(AC|BD) = 0, \quad 
        \tilde \theta(AD|BC) = 0,
    \end{equation}
    To prove $E_{\theta}(\varphi)<E^{\theta}(\varphi)$, it thus suffices to prove
    \begin{equation}
        E_{\theta}(\varphi)=\kappa E_{\tilde\theta}(\psi_{\frac{1}{3}}), \qquad E^{\theta}(\varphi) =\kappa E^{\tilde\theta}(\psi_{\frac{1}{3}}),
    \end{equation}
    for some distribution $\tilde \theta$ and normalization constant $\kappa >0$.
    Indeed, $\tilde \theta$ and $\kappa$ are chosen such that
    \begin{align}
        \begin{split}
            \kappa\tilde \theta(A|BCD)&=\sum \theta(AX|BCD X'),\\
            \kappa\tilde \theta(B|ACD)&=\sum \theta(BX|ACD X'),\\
            \kappa\tilde \theta(C|ABD)&=\sum \theta(CX|ABD X'),\\
            \kappa\tilde \theta(D|ABC)&=\sum \theta(DX|ABC X'),\\
            \kappa\tilde \theta(AB|CD)&=\sum \theta(ABX|CD X'),
        \end{split}
    \end{align}
    where each sum is over all subsets $X \subset [k]\setminus\{A,B,C,D\}$ and $X'$ the corresponding complement.
    Note that $\kappa$ and $\tilde \theta$ are then uniquely fixed by requiring $\tilde \theta$ sum to one (and also noting that at least $\tilde \theta(AB|CD) > 0$ so $\kappa > 0$ is well-defined).
\end{proof}

\subsection{Quantitative separation}
\label{subsec:quantitative-separation}

In this section, we improve upon~\cref{thm:upper lower separation general k} by constructing a \textit{quantitative} upper bound for the lower quantum functional which enables us to determine, for some tensors, the magnitude of the separation $E_{\theta}(\psi) < E^{\theta}(\psi)$. 
For simplicity we focus on~$k=4$, and consider only the~$\theta$ given by~$\theta(AB|CD) = 1$ (all other weights are zero).
In this case, we write $E_{AB}(\psi) = E_\theta(\psi)$ for the logarithmic lower quantum functional and $E^{AB}(\psi) = E^\theta(\psi)$ for the logarithmic upper quantum functional evaluated on a $4$-tensor $\psi$.

We first recall a definition from~\cite{burgisserTheoryNoncommutativeOptimization2019}\footnote{This notion can be traced back to \cite{gurvitsClassicalComplexityQuantum2004}, where it was used in the context of Edmonds' problem.}, which helps us state a clean bound. Let~$V = \C^{d_1} \ot \dotsb \ot \C^{d_k}$, and let~$\SL = \SL_{d_1} \otimes \dotsb \otimes \SL_{d_k}$ act on~$V$ via the tensor-product action.
Then the~\emph{capacity} of~$\psi \in V \setminus \{0\}$ is defined by
\[
    \capa(\psi) \coloneq \inf_{g \in \SL} \norm{g \psi}.
\]
Recall that a tensor~$\psi$ is called semistable if~$\capa(\psi) > 0$ (equivalently, $0 \not\in \overline{\SL \cdot \psi}$).
By~\cref{thm:kempf-ness-general}, $\capa(\psi) = \norm{\psi}$ if and only if~$\mu(\psi) = 0$.

Let~$\psi \in \C^{d_A} \ot \C^{d_B} \ot \C^{d_C} \ot \C^{d_D} \setminus \{0\}$ with~$d \coloneq d_A d_B = d_C d_D$.
Then the flattening~$\psi_{AB|CD}$ can be viewed as a~$d \times d$ matrix.
It turns out that its determinant provides quantitative information constraining~$E_{AB}$.
\begin{lemma}
    \label{lem:det bound}
    Assume~$\capa(\psi) > 0$.
    Then for all~$\varphi \in \overline{\SL \cdot \psi}$,
    \[
        \frac{\abs{\det(\varphi_{AB|CD})}^2}{\norm{\varphi}^{2 d}} \leq \frac{\abs{\det(\psi_{AB|CD})}^2}{\capa(\psi)^{2d}}.
    \]
\end{lemma}
\begin{proof}
    The determinant is an~$\SL$-invariant, so~$\det(\varphi_{AB|CD}) = \det(\psi_{AB|CD})$.
    Moreover if~$\varphi \in \overline{\SL \cdot \psi}$ then~$\norm{\varphi} \geq \capa(\psi)$ (by definition of the capacity, and continuity of the norm).
\end{proof}

This lemma finds its use as follows: 
recall the logarithm of the lower quantum functional~$E_{AB}(\psi)$ is the maximum achievable entropy~$H_{AB}(\rho_{AB})$ over all $AB$-marginals, $\rho_{AB}$, of (normalizations of) tensors $\varphi \in \overline{\SL \cdot \psi}$.
The quantities in \cref{lem:det bound} correspond to the determinants of such marginals, i.e.,
\[
    \det(\rho_{AB}) = \frac{\abs{\det(\varphi_{AB|CD})}^2}{\norm{\varphi}^{2d}}.
\]
Therefore, we can derive an upper bound on the lower quantum functional $E_{AB}(\psi)$ by relaxing the orbit-membership constraint and simply compute the maximum possible entropy $H_{AB}(\eta)$ over all $4$-tensors $\eta$ which satisfy the determinant bound only:
\begin{corollary}
  \label{cor:Edetab bound}
For semistable~$\psi$, let ~$c_\psi = \frac{\abs{\det(\psi_{AB|CD})}^2}{\capa(\psi)^{2d}}$. Then $E_{AB}(\psi) \leq \tilde E^{\det}_{AB}(c_\psi)$ where
\begin{equation}
     \tilde E^{\det}_{AB}(c) \coloneq\max\{ H_{AB}(\eta) \,|\, \eta \in \C^{d_A} \ot \C^{d_B} \ot \C^{d_C} \ot \C^{d_D}, \, \norm{\eta}=1, \abs{\det(\eta_{AB|CD})}^2 \leq c \}.
\end{equation}
\end{corollary}
\noindent
We determine the value of $\tilde E^{\det}_{AB}(c)$ exactly in~\cref{prop:Edetab-exact-value}.
We also exhibit a continuous family~$\mathsf{Q}_\gamma$ of tensors in~$(\C^2)^{\ot 4}$ such that~$E_{AB}(\mathsf{Q}_\gamma) = \tilde E^{\det}_{AB}(c_{\mathsf{Q}_\gamma})$ in~\cref{sec:werner-gamma}, showing that the bound can be tight.

Note the dimensions of a tensor $\psi \in \mathbb C^{d_{A}} \ot \mathbb C^{d_{B}} \ot \mathbb C^{d_{C}} \ot \mathbb C^{d_{D}}$ impose constraints on the maximum possible $AB|CD$ determinant.
When $d=d_Ad_B=d_Cd_D$, the value of $c_{\psi}$ is always bounded by $c_\psi \leq 1/d^d$ and consequently $\tilde E^{\det}_{AB}(c_\psi) \leq \log_2(d)$.
Importantly, whenever the first inequality is strict, so is the second:
\begin{lemma}
    \label{lem:strict-det-ineq-implies-strict-Edet-ineq}
    For~$\psi \in \C^{d_A} \ot \C^{d_B} \ot \C^{d_C} \ot \C^{d_D}$ and~$d \coloneqq d_A d_B = d_C d_D$, $\tilde E^{\det}_{AB}(c_\psi) < \log_2 (d)$ whenever~$c_\psi < \frac{1}{d^d}$.
\end{lemma}
\begin{proof}
    Suppose~$\tilde E^{\det}_{AB}(c_\psi) = \log_2(d)$. Then there exists some $\eta \in \C^{d_A} \ot \C^{d_B} \ot \C^{d_C} \ot \C^{d_D}$ such that~$\norm{\eta}=1$ and~$\rho_{AB} = \eta_{AB} \eta_{AB}^*$ satisfies $H(\rho_{AB}) = \log_2(d)$ and $\det(\rho_{AB}) = \abs{\det(\eta_{AB|CD})}^2 \leq c_\psi$.
    Since~$\rho_{AB}$ is a~$d \times d$ matrix, $H(\rho_{AB}) = \log_2(d)$ implies its eigenvalues are all~$1/d$, hence their product is~$1/d^d$ and $c_\psi \geq 1/d^{d}$.
\end{proof}

\begin{lemma}
\label{lem:semistable det upper bound gives opens and dense}
    Assume~$d_A = d_B = d_C = d_D \geq 2$, $d = d_A d_B$ ($= d_C d_D$) and~$c > 0$.
    The set~$U_c \subseteq V = \C^{d_A} \ot \C^{d_B} \ot \C^{d_C} \ot \C^{d_D}$ consisting of~$\psi$ which are semistable, such that~$c_\psi < c$, is non-empty and contains a Euclidean-open subset.
    Moreover, when~$c = 1/d^d$ the set $U_c$ is Euclidean-dense.
\end{lemma}
\begin{proof}
    Let~$V^s$ denote the set of \emph{stable} tensors in~$V$, that is, those with closed~$\SL$-orbit and discrete stabilizer (in particular, it cannot be the zero tensor).
    For~$d_A = d_B = d_C = d_D$, a generic tensor is stable~\cite{bryanExistenceLocallyMaximally2018a}, as the dimension of the geometric invariant theory quotient~$\Proj(V) // \SL$ is the expected dimension~$\dim \Proj(V) - \dim \SL$ (this relies on the specific format of the tensors).

    Let~$\mu$ denote the moment map from~\cref{prop:moment map}.
    Then by~\cite[Thm.~7.4, Rmk.~7.5]{kirwanCohomologyQuotientsSymplectic1984} the natural map~$f\colon \mu^{-1}(0)/\SU \to (\SL \cdot \mu^{-1}(0))/\SL$ is a homeomorphism.
    The stable tensors are naturally a subset of~$\SL \cdot \mu^{-1}(0)$, and so the map~$\psi \mapsto c_\psi$ is the composition of the inclusion~$V^s \to \SL \cdot \mu^{-1}(0)$, the quotient map~$\SL \cdot \mu^{-1}(0) \to (\SL \cdot \mu^{-1}(0)) / \SL$, $f^{-1}$, and then the map
    \[
        [[\psi]_{\Proj(V)}]_{\SU} \mapsto \frac{\abs{\det(\psi_{AB|CD})}^2}{\norm{\psi}^{2d}},
    \]
    where~$[\psi]_{\Proj(V)} \in \mu^{-1}(0)$ and~$[\cdot]_{\SU}$ denotes its equivalence class in~$\mu^{-1}(0)/\SU$.
    All these maps are continuous, hence so is~$\psi \mapsto c_\psi$ on~$V^s$.
    By continuity, the sublevel sets~$\{\psi : c_\psi < c\}$ of~$\psi \mapsto c_\psi$ are Euclidean-open; they are clearly also non-empty whenever~$c > 0$, and Euclidean-dense when~$c = 1/d^d$ is the maximum value of~$c_\psi$, as~$V^s$ is dense in~$V$.
\end{proof}
This lemma helps us establish that the separation between upper and lower quantum functionals holds generically, for balanced order-$4$ tensors:
\begin{corollary}
\label{cor:generic-separation}
    Let~$d_A = d_B = d_C = d_D \geq 2$.
    Then the set of all~$\psi \in V$ with~$F_\theta(\psi) < F^\theta(\psi)$ for all~$\theta \in \distsnc^4 \setminus \distssing^4$ contains a Euclidean-open and dense subset of~$V$.
\end{corollary}
\begin{proof}
    First fix a~$\theta \in \distsnc^4 \setminus \distssing^4$.
    Assume without loss of generality that~$\theta(AB|CD) \geq 0$ and~$ \theta(AC|BD) = \theta(AD|BC) = 0$.
    Let~$d = d_A d_B$ ($= d_C d_D$) as before, and set~$d' = d_A$ so that~$(d')^2 = d$.
    
    The set of~$\psi$ with~$E^\theta(\psi) = (\theta_A + \theta_B + \theta_C + \theta_D) \cdot \frac12 \log_2(d) + \theta(AB|CD) \log_2(d)$ contains a Euclidean-open and Euclidean-dense subset.
    Indeed, for~$\lambda = (d', \dotsc, d') \in \Z^{d'}$ and~$\nu = (1, \dotsc, 1) \in \Z^{d}$, the product of projectors 
    \[
        P = P^{AB|CD}_\nu P^A_\lambda P^B_\lambda P^C_\lambda P^D_\lambda
    \]
    is non-zero (in fact it is equal to~$P^{AB|CD}_\nu$, by the argument given in the proof of~\cref{prop:psi p gives separation}). 
    The set of~$\psi$ such that~$P \psi^{\kron d} \neq 0$ is a Euclidean-open and Euclidean-dense subset (it is even Zariski-open), and $F^\theta(\psi)$ achieves the claimed value for such~$\psi$.
    
    On the other hand, by \cref{lem:semistable det upper bound gives opens and dense} for a Euclidean-open and dense set of~$\psi$, we also have~$c_\psi < 1/d^d$, and hence
    \[
        \log_2 F_\theta(\psi) = E_\theta(\psi) \leq \sum_{b\in \supp \theta} \theta_b E_b(\psi) < (\theta_A + \theta_B + \theta_C + \theta_D) \cdot \frac12 \log_2(d) + \theta(AB|CD) \log_2(d).
    \]
    since~$\tilde E^{\det}_{AB}(c_\psi) < \log_2(d)$.

    As the intersection of two Euclidean-open and dense subsets is open and dense, we obtain that~$F_\theta(\psi) < F^\theta(\psi)$ on a Euclidean-open and dense set~$U_\theta$.

    Now note that the construction of~$U_\theta$ depends only on~$\supp \theta$. Selecting one~$\theta$ for every possible laminar but non-singleton support, we see that~$U = \cap_\theta U_\theta$ is Euclidean-open and dense, as it is a finite intersection of open and dense sets.
\end{proof}
In~\cite{christandlAsymptoticTensorRank2025} it was shown that universal spectral points~$F$ have Zariski-closed sublevel sets~$F^{-1}((-\infty, r])$. If~$r$ is not the maximal value of~$F$, these sublevel sets will automatically be lower-dimensional, so cannot contain a non-empty Euclidean-open subset. Combining this with the above quantitative bound, we see that the lower quantum functional is never a universal spectral point when~$\theta \in \dists^k \setminus \distssing^k$:
\begin{theorem}
    \label{prop:general-lower-func-not-spectral}
    For all~$k \geq 4$ and~$\theta \in \dists^k \setminus \distssing^k$, the lower quantum functional~$F_\theta$ is not a spectral point for~$k$-tensors, i.e., $F_\theta \not\in \asympspec(\T^k)$.
\end{theorem}
\begin{proof}
    Let~$b = \{S, \overline{S}\}$ be a non-singleton bipartition of~$[k]$.
    Select distinct~$A, B \in S$ and~$C, D \in \overline{S}$, and consider~$d_{A} = d_{B} = d_{C} = d_{D} = 3$ and~$d_i = 1$ for all other~$i$.
    Define
    \begin{equation}\label{eq:AME-state}
        \mathsf{L} = \frac13 \sum_{\ell, \ell' = 1}^3 (e_{\ell} \ot e_{\ell'} \ot e_{\ell + \ell'} \ot e_{\ell + 2 \ell'})_{ABCD} \ot (e_1 \ot \dotsb \ot e_1)_{[k] \setminus \{A,B,C,D\}}.
    \end{equation}
    Then~$\mathsf{L}$ is \emph{absolutely maximally entangled} (AME) on $ABCD$~\cite[Sec.~III.B]{goyenecheAbsolutelyMaximallyEntangled2015}, i.e., the eigenvalues of its marginal along any bipartition are all uniform.
    This guarantees that the maximum value of~$E_\theta = \log_2 F_\theta$ on~$V = \C^{d_1} \ot \dotsb \ot \C^{d_k}$ is equal to~$E_\theta(\mathsf{L}) = M^\theta(\mathsf{L}) = \sum_{b' \in \supp \theta} \log_2 R_{b'}(\mathsf{L}) = \sum_{b' \in \supp \theta} \theta_{b'} E_{b'}(\mathsf{L})$.

    The set~$U \subseteq V$ of semistable~$\psi \in V$ with~$\abs{\det(\psi_{AB|CD})}^2/\capa(\psi)^{18} < 0.9 \cdot 1/3^3$ is non-empty, and contains a Euclidean-open subset of~$V$.
    These all satisfy~$E_{AB}(\psi) \leq c' \coloneq \tilde E^{\det}_{AB}(0.9 \cdot 1/3^3) < \log_2(3)$.
    Every~$\psi \in U$ then satisfies~$E_\theta(\psi) \leq \sum_{b \in \supp \theta} \theta_{b'} E_{b'}(\psi) \leq \sum_{b' \in \supp \theta} \theta_{b'} E_{b'}(\mathsf{L}) + \theta_b (\log_2(3) - c')$.
    The Zariski-closure of~$U$ is all of~$V$, since it contains a Euclidean-open subset of~$V$.
    If~$F_\theta = 2^{E_\theta}$ were a spectral point, it would be Zariski-lower-semicontinuous~\cite{christandlAsymptoticTensorRank2025};
    that is, it would satisfy that for every $r \in \R_{\geq 0}$, the sublevel set $F_\theta^{-1}((-\infty, r])$ is Zariski-closed.
    But~$F_\theta$ is at most~$2^{M^\theta(\mathsf{L}) - \theta_b}$ on~$U$, yet attains~$2^{M^\theta(\mathsf{L})}$ on~$\mathsf{L} \in V$.
    This is a contradiction.
\end{proof}

\begin{remark}
\label{remark:lower-not-spectral-via}
    To prove that~$F_\theta$ is not a spectral point for~$\theta \in \distsnc \backslash \distssing$, one may alternatively use the exact calculation of~$\tilde E^{\det}_{AB}(c)$ and the continuous family of tensors~$\mathsf{Q}_\gamma$ discussed in~\cref{sec:werner-gamma}. This family attains a continuum of values for the~$F_\theta$, which is also impossible for Zariski-lower-semicontinuous functionals, see~\cite{christandlAsymptoticTensorRank2025}.
    This approach does not obviously generalize to~$\theta \in \dists \setminus \distsnc$ as we do not know how to evaluate the lower quantum functional exactly in this case.
\end{remark}

\section{Agreement of upper and lower laminar functionals}\label{sec:new-functionals}

\subsection{Laminar-quantum-robust tensors}

In \cref{sec:separation}, we proved upper and lower quantum functionals of laminarly-supported distributions do not generally agree for higher-order tensors. 
Therefore, unlike the case for $3$-tensors, quantum functionals for higher-order tensors are not necessarily \textit{universal} spectral points.
In this section, we prove that quantum functionals of laminarly-supported distributions are nevertheless spectral points in the asymptotic spectrum of a particular sub-semiring of higher-order tensors, which we define now.

\begin{definition}
    \label{def:Tnck}
    A tensor $\psi \in \T^k$ is said to be laminar-quantum-robust if for all~$\theta \in \distsnc$, $F^\theta(\psi) = F_\theta(\psi)$.
    The set of laminar-quantum-robust tensors is denoted by $\Tnc^k$.
\end{definition}

\begin{proposition}
    \label{prop:Tnck is a semiring}
    The laminar-quantum-robust tensors~$\Tnc^k$ form a sub-semiring of~$\T^k$.
    Moreover, for every~$\theta \in \distsnc$, the restriction of $F^\theta$ (or equivalently $F_\theta$) to $\Tnc^k$ is in~$\asympspec(\Tnc^k)$.
\end{proposition}
\begin{proof}
    First of all, we observe that the unit tensors are in~$\Tnc^k$.
    Clearly~$F_\theta(\unitk{n}{k}) = n$ for all~$n \geq 1$, and~$F^\theta(\unitk{n}{k}) \leq n$ by~\cref{prop:E_upper_M_bound}.
    Next, suppose~$\psi,\varphi \in \Tnc^k$ and let~$\theta \in \distsnc$.
    By superadditivity of~$F^\theta$ and subadditivity of~$F_\theta$, we obtain
    \[
        F^\theta(\psi) + F^\theta(\varphi) 
        \geq F^\theta(\psi \oplus \varphi)
        \geq F_\theta(\psi \oplus \varphi)
        \geq F_\theta(\psi) + F_\theta(\varphi).
    \]
    Since~$\psi,\varphi \in \Tnc^k$, these inequalities are all equalities, and hence~$\psi \oplus \varphi \in \Tnc^k$.
    The proof that~$\psi \kron \varphi \in \Tnc^k$ follows similarly.
    Therefore~$\Tnc^k$ contains the unit tensors, is closed under addition and multiplication, and hence a sub-semiring of~$\T^k$.
    The claim that~$F^\theta = F_\theta$ is a spectral point of~$\Tnc^k$ now follows directly from~\cref{prop:upper qfunc properties,prop:lower qfunc properties}.
\end{proof}

An extension result of Strassen \cite[Prop.~3.1.1]{Strassen1988spectrum} (see also~\cite[Thm.~4.16]{wigdersonAsymptoticSpectraTheorya}) implies the existence of universal spectral points which agree with~$F^\theta=F_\theta$ on~$\Tnc^k$:
\begin{proposition}
  \label{prop:Ftheta admits extension}
  For every~$\theta \in \distsnc^k$, there exists a spectral point~$F \in \asympspec(\T^k)$ such that~$ F(\psi) = F^\theta(\psi) = F_\theta(\psi)$ for all~$\psi \in \Tnc^k$.
\end{proposition}
Recall from~\cref{prop:general-lower-func-not-spectral} that the lower quantum functional $F_{\theta}$ (for any $\theta \in \distsnc^k\setminus\distssing^k$) is \textit{not} a spectral point for all $k$-tensors and therefore $F \neq F_\theta$. 
Nevertheless, it is possible that the upper quantum functional $F^{\theta}$ is a spectral point for all $k$-tensors and thus it is possible that $ F = F^\theta$.

Although the above establishes existence of spectral points extending~$F^\theta = F_\theta$ from~$\Tnc^k$ to~$\T^k$, it is currently unclear whether these spectral points provide new information about asymptotic restrictions.
Concretely, we pose the following problem:
\begin{problem}
    Do there exist tensors~$\psi, \varphi \in \Tnc^k$ such that $F(\psi) \leq F(\varphi)$ for all previously known spectral points $F$, but nevertheless there exists a distribution~$\theta \in \distsnc^k$ such that $F^\theta(\psi) > F^\theta(\varphi)$?
\end{problem}
By known spectral points, we mean singleton quantum functionals, or spectral points obtained by grouping parties together and then applying singleton quantum functionals.

We now turn to the problem of finding examples of laminar-quantum-robust tensors.
\begin{remark}
    \label{remark:AME-states-are-in-Tnc}
    Let~$\psi \in (\C^d)^{\ot k}$ be an \emph{absolutely maximally entangled state}, i.e., for all~$\emptyset \neq S \subsetneq [k]$ with~$\abs{S} \leq k/2$, one has~$\rho_S = \psi_S \psi_S^* / \norm{\psi}^2 = I_{d^s} / d^s$.
    Then clearly~$H_\theta(\psi) = M^\theta(\psi) \leq E_\theta(\psi) \leq E^\theta(\psi) \leq M^\theta(\psi)$ for all~$\theta \in \distsnc^k$ (by \cref{prop:E_upper_M_bound}), and so~$\psi \in \Tnc^k$.
    Such tensors are well-studied in quantum information theory, see e.g.~\cite{higuchiHowEntangledCan2000a,helwigAbsoluteMaximalEntanglement2012,goyenecheAbsolutelyMaximallyEntangled2015}, and~\cite{rajchel-mieldziocAbsolutelyMaximallyEntangled2026} for a recent survey.
    
    More generally, if~$\psi \in (\C^d)^{\otimes k}$ is such that for every~$S \subsetneq [k]$, $\psi_S \psi_S^*$ is proportional to an orthogonal projector, then $M^\theta(\psi) = H^\theta(\psi) = E_\theta(\psi) = E^\theta(\psi)$.
    A canonical example is the \emph{determinant} tensor $\det_k = e_1 \wedge \dotsb \wedge e_k \in (\C^k)^{\otimes k}$.
\end{remark}

Our next result provides more examples of tensors which belong to the sub-semiring $\Tnc^k$ of $k$-tensors.
\begin{theorem}\label{thm:functionals-for-3structures}
    Any embedding of a unit tensor of order $\ell<k$, i.e. $\unitk{r}{\ell}\otimes e_1^{\otimes k-\ell}$, all the embeddings of $3$-tensors, and the embeddings of $\W_4 = e_2 \ot e_1 \ot e_1 \ot e_1 + e_1 \ot e_2 \ot e_1 \ot e_1 + e_1 \ot e_1 \ot e_2 \ot e_1 + e_1 \ot e_1 \ot e_1 \ot e_2$  as $k$-tensors, are in $\Tnc^k$. 
\end{theorem}

We prove \cref{thm:functionals-for-3structures} separately for embedded tensors in~\cref{lem:embedded-tensors} and for the~$\W_4$ in~\cref{thm:W4-non-crossing-agree}.

\begin{remark}
    As $\Tnc^k$ is a semiring, \cref{thm:functionals-for-3structures} implies that $\Tnc^k$ includes entanglement structures as in \cite[Def. 11]{christandl2020tensornetworksreps} and \cite{Christandl2024resourcetensors}, as long as the hyperedges of order at least $4$ use AME states, $\W_4$ or $\unitk{r}{\ell}$. It also includes direct sums of such entanglement structures.
\end{remark}

\subsection{Embedded tensors are in \texorpdfstring{$\Tnc^k$}{Tnc}}

We first establish that the upper and lower quantum functionals can be easily evaluated on embedded tensors, by appropriately redistributing the weight on the bipartitions. A common ingredient is the following definition:
\begin{definition}
    Let~$1 \leq \ell < k$.
    For a distribution $\theta\in\dists^k$ define its restricted distribution to $[\ell]$ by 
    \[
        \theta'_{S} \coloneq\frac{1}{C_{\theta,\ell}}\sum_{\substack{\{T,\overline{T}\}\in \bipartitions^k \\ T\cap[\ell]=S}}\theta_T
    \]
    with the normalization factor
    \[
        C_{\theta,\ell} \coloneq\sum_{\substack{\{T,\overline{T}\}\in \bipartitions^k\\ \emptyset\neq T \cap[\ell]\subsetneq[\ell]}}\theta_T.
    \]
\end{definition}
\begin{remark}
    Note that if $\theta$ is laminar, then so is $\theta'$: the support of $\theta'$ is the restrictions of the bipartitions in the support of $\theta$ to bipartitions of $[\ell]$, and weak inclusion relations are preserved by an intersection with another subset (if $S_1\subseteq S_2$ then $S_1\cap[\ell]\subseteq S_2\cap[\ell]$). 
\end{remark}
\begin{lemma}
    \label{lem:lower qfunc on embedding}
    Let~$\psi \in \C^{d_1} \ot \dotsb \ot \C^{d_\ell}$, let~$\theta \in \dists^k$ and~$\theta'$ its restricted distribution on~$[\ell]$.
    Then
    \[
        E_\theta(\psi \ot e_1^{\ot (k-\ell)}) = C_{\theta,\ell} E_{\theta'}(\psi).
    \]
\end{lemma}
\begin{proof}
    Recall that for a~$k$-tensor~$\varphi \in \C^{d_1} \ot \dotsb \ot \C^{d_k}$, 
    \[
        E_\theta(\varphi) = \sup_{\substack{A_i\in \GL_{d_i}(\C)\\i\in[k]}}\sum_{\{S,\overline{S}\}\in \supp \theta}\theta_{S} H_S((A_1 \ot \dotsb \ot A_k) \varphi),
    \]
    where
    \begin{align*}
        H_S(\varphi) = H(\frac{\varphi_S \varphi_S^*}{\norm{\varphi}^2}).
    \end{align*}
    If~$\varphi = (A_1 \ot \dotsb \ot A_k) (\psi \ot e_1^{\ot (k-\ell)})$, then the~$S$-flattening~$\varphi_S$ of~$\varphi$ satisfies
    \[
        \varphi_S \varphi_S^* = (\bigotimes_{i \in S \cap [\ell]} A_i) \psi_S (\bigotimes_{i \in \overline{S} \cap [\ell]} A_i^* A_i) \psi_S^* (\bigotimes_{i \in S \cap [\ell]} A_i^*) \otimes \bigotimes_{i = \ell+1, i \in S}^k A_i e_1 e_1^* A_i^* \cdot \prod_{i = \ell+1, i \in \overline{S}}^k e_1^* A_i^* A_i e_1
    \]
    Note that the $e_1 A_i^* A_i e_1$ are just scalars.
    The first part is equal to~$\xi_{S \cap [\ell]} \xi_{S \cap [\ell]}^*$ where~$\xi = (\ot_{i \in [\ell]} A_i) \psi$.
    The second part is a tensor product of (unnormalized) rank~$1$ projectors, and hence does not affect the non-zero eigenvalues of the normalized operator.
    The final rescaling factor can be ignored for computing~$H_S$, as it gets absorbed into the normalization.
    We conclude that $H_S(\varphi) = H_{S \cap [\ell]}(\xi)$.
    Therefore
    \begin{align*}
        E_\theta(\psi \ot e_1^{\otimes k-\ell}) 
        & = \sup_{\substack{A_i\in \GL_{d_i}(\C)\\i\in[k]}} \sum_{\{S,\overline{S}\}\in \supp \theta}\theta_{S} H_S((A_1 \ot \dotsb \ot A_k) (\psi \ot e_1^{\ot k-\ell})) \\
        & = \sup_{\substack{A_i\in \GL_{d_i}(\C)\\i\in[\ell]}} \sum_{\{S,\overline{S}\}\in \supp \theta}\theta_{S} H_{S \cap [\ell]}((A_1 \ot \dotsb \ot A_\ell) \psi) \\
        & = \sup_{\substack{A_i\in \GL_{d_i}(\C)\\i\in[\ell]}} \sum_{\{S,\overline{S}\}\in \supp \theta'} \big(\sum_{\substack{\{T, \overline{T}\} \in \supp \theta \\ T \cap [\ell] = S} }\theta_{S}\big) H_S((A_1 \ot \dotsb \ot A_\ell) \psi) \\
        & = C_{\theta,\ell} \sup_{\substack{A_i\in \GL_{d_i}(\C)\\i\in[\ell]}} \sum_{\{S,\overline{S}\}\in \supp \theta'} \frac{1}{C_{\theta,\ell}} \big(\sum_{\substack{\{T, \overline{T}\} \in \supp \theta \\ T \cap [\ell] = S} }\theta_{S}\big) H_S((A_1 \ot \dotsb \ot A_\ell) \psi) \\
        & = C_{\theta,\ell} \sup_{\substack{A_i\in \GL_{d_i}(\C)\\i\in[\ell]}}\sum_{\{S,\overline{S}\}\in \supp \theta'} \theta'_S H_S((A_1 \ot \dotsb \ot A_\ell) \psi) = C_{\theta,\ell} E_{\theta'}(\psi). \qedhere
    \end{align*}
\end{proof}
\begin{lemma}
    \label{lem:upper qfunc on embedding}
    Let~$\psi \in \C^{d_1} \ot \dotsb \ot \C^{d_\ell}$, let~$\theta \in \distsnc^k$ and~$\theta'$ its restricted distribution on~$[\ell]$.
    Then
    \[
        E^\theta(\psi \ot e_1^{\ot (k-\ell)}) = C_{\theta,\ell} E^{\theta'}(\psi).
    \]
\end{lemma}
\begin{proof}
    We need a tool analogous to ignoring tensor-product vectors for the representational expressions these functionals use.
    In other words, we use the fact that a tensor-product factor on one of the local parties can be ignored: by \cref{eq:proj-restricted-char-form}, if $e_1$ is a non-zero vector on one of the local parties, w.l.o.g. on the last one $k$, $\emptyset\neq S\subsetneq[k]$ is a nontrivial subset of the parties, and $\lambda\vdash n$ is a partition, we first show that the bipartition $\{S,\overline{S}\}=\{[k-1],\{k\}\}$ can be ignored. For this bipartition, by the second equality of \cref{lem:complementaryProjectionsOnSymmetric}, $P^{V_S}_\lambda(\psi \otimes e_1)^{\kron n} = P^{V_{k}}_\lambda(\psi \otimes e_1)^{\kron n}= \psi^{\kron n} \otimes (P_\lambda e_1^{\kron n})$ which is non-zero only for $\lambda=(n)$.
    Therefore, this projector does not change the state on one hand, and on the other hand the entropy~$H(\lambda/n) = 0$, so it can be ignored for evaluating the upper  quantum functional.
    Next, suppose $S\setminus\{k\}$ is a non-trivial subset of $[k-1]$, in which case 
    \begin{align*}
        P^{V_S}_\lambda (\psi \otimes e_1)^{\kron n}
        & = \frac{\dim([\lambda])}{n!}\sum_{\pi\in S_n}\chi^*_\lambda(\pi)\tau_S(\pi)\otimes\tau_{\overline{S}}(e)(\psi^{\kron n}\otimes e_1^{\kron n}) \\
        & = \big(\frac{\dim([\lambda])}{n!}\sum_{\pi\in S_n}\chi^*_\lambda(\pi)\tau_{S\setminus\{k\}}(\pi)\otimes\tau_{\overline{S}\setminus\{k\}}(e)(\psi^{\kron n})\big)\otimes e_1^{\kron n}.
    \end{align*}
    Here~$\tau_S$ denotes the tensor permutation action on~$V_S^{\ot n}$.
    The equality holds because~$e_1^{\kron n}$ is permutation-invariant.
    Now we can peel off the single-party tensor-products one-by-one to ignore all of them for one projector.
    Then we can propagate our ignoring to the next operators one-by-one to see that $\bigg(\prod_{\{S,\overline{S}\}\in \supp \theta}P^{V_S}_{\lambda_S}\bigg)(\psi\otimes e_1^{\otimes k-\ell})^{\kron n} \neq 0$ iff $\bigg(\prod_{\{S,\overline{S}\}\in \supp \theta} P^{V_{S\cap[\ell]}}_{\lambda_S}\bigg)(\psi^{\kron n}) \neq 0$, which is the tool we use in the chain of equalities below:
    \begin{align*}
        & E^\theta(\psi\otimes e_1^{\otimes k-\ell}) \\
        & =\sup\{\sum_{(\lambda_{S})_{\{S,\overline{S}\}\in \supp \theta}} \theta_S H(\tfrac{1}{n}\lambda_S) : n\in\N, \lambda_S\vdash n, \bigg(\prod_{\{S,\overline{S}\}\in \supp \theta}P^{V_S}_{\lambda_S}\bigg)(\psi\otimes e_1^{\otimes k-\ell})^{\kron n}\neq0\}\\
        & =C_{\theta,\ell}\cdot\sup\{\sum_{\substack{(\lambda_{S})_{\{S,\overline{S}\}\in \supp \theta}\\ \emptyset\neq S\cap[\ell]\subsetneq[\ell]}}\frac{\theta_S}{C_{\theta,\ell}}H(\tfrac{1}{n}\lambda_S):n\in\N,\lambda_S\vdash n, \bigg(\prod_{\substack{\{S,\overline{S}\}\in \supp \theta\\ \emptyset\neq S\cap[\ell]\subsetneq[\ell]}}P^{V_{S\cap[\ell]}}_{\lambda_S}\bigg)(\psi^{\kron n})\neq0\}\\
        & =C_{\theta,\ell}\cdot\sup\{\sum_{(\lambda_{S})_{\{S,\overline{S}\}\in \supp \theta'}}\underbrace{\frac{1}{C_{\theta,\ell}}\bigg(\sum_{\substack{\{T,\overline{T}\}\in \supp \theta\\ T\cap[\ell]=S}}\theta_T\bigg)}_{=\theta'_S}H(\tfrac{1}{n}\lambda_S):n\in\N,\lambda_S\vdash n, \bigg(\prod_{\{S,\overline{S}\}\in \supp \theta'}P^{V_S}_{\lambda_S}\bigg)(\psi^{\kron n}) \neq 0\}\\
        & =C_{\theta,\ell}E^{\theta'}(\psi). \qedhere
    \end{align*}
\end{proof}
\begin{proposition}[Laminar functionals agree on tensor embeddings]\label{lem:embedded-tensors}
  Let~$1 \leq \ell < k$.
  Let $\psi\in\C^{c_1}\otimes\cdots\otimes\C^{c_\ell}$ be an $\ell$-tensor.
  Assume that for every laminarly-supported distribution over the bipartitions of the $\ell$ parties $\theta\in\distsnc^\ell$ the (logarithmic) upper and lower quantum functionals agree on $\psi$: $E^\theta(\psi)=E_\theta(\psi)$. 
  Then, for any~$\theta \in \distsnc^k$ the (logarithmic) upper and lower quantum functionals agree on the embedded tensor~$\tilde\psi = \psi \ot e_1^{\ot (k-\ell)}$: $E^\theta(\tilde\psi) = E_\theta(\tilde \psi)$.
  In other words, any embedding~$\tilde\psi$ of~$\psi \in \Tnc^\ell$ as a~$k$-tensor is in~$\Tnc^k$.
\end{proposition}
\begin{proof}
    This follows directly from~\cref{lem:lower qfunc on embedding,lem:upper qfunc on embedding}.
\end{proof}

\cref{lem:embedded-tensors} combined with the fact that $\Tnc$ is a semiring (\cref{prop:Tnck is a semiring}) implies that the iterated matrix multiplication, the dome tensor~\cite[\S 4.1]{Christandl2019tensorSurgery}, and more generally any graph tensor as in~\cite[\S 2.1]{Christandl2019tensorSurgery} (see also \cite{Christandl2019AsymptoticRankGraphTensors}) are in $\Tnc$.

As a very particular case, the same normalization of upper and lower laminar quantum functionals on unit tensors, guarantees that they also agree after embedding.
Here we can very explicitly evaluate the upper and lower quantum functionals, and see that they agree:
\begin{lemma}
    \label{lem:ncfunc on subset units}
    Let~$\emptyset \neq S \subsetneq [k]$  and~$\theta \in \distsnc$.
    Consider ``the unit tensor of rank $2$''\footnote{This is not ``rank 2'' when~$\abs{S}=1$, but the formula evaluates to~$0$ as expected (it is equivalent to the rank~$1$ unit tensor).} on~$S$,
    \[
        \unitk{2}{S} = \sum_{i=1}^2 (e_i \ot \dotsb \ot e_i)_{S} \ot (e_1 \ot \dotsb \ot e_1)_{\overline{S}}.
    \]
    Then
    \[
        E^\theta(\unitk{2}{S}) = E_\theta(\unitk{2}{S}) = 1 - \sum_{\substack{b = \{T, \overline{T}\} \in \supp \theta \\ S \subseteq T \, \lor \, S \subseteq \overline{T}}} \theta_b.
    \]
\end{lemma}
\begin{proof}
    Recall from~\cref{prop:E_upper_M_bound} that~$E_\theta(\psi) \leq E^\theta(\psi) \leq M^{\theta}(\psi)$ where~$M^{\theta}(\psi) = \sum_{b \in \supp \theta} \theta_b \log_2 R_b(\psi)$.
    The flattening ranks of the unit tensor on~$S$ are given by
    \[
        R_{\{ T, \overline{T} \}}(\unitk{2}{S}) = \begin{cases}
            1 & \text{if }  S \subseteq T \text{ or } S \subseteq \overline{T}, \\
            2 & \text{otherwise.}
        \end{cases}
    \]
    This directly provides the right upper bound on~$E^\theta(\unitk{2}{S})$.
    Similarly, $E_\theta(\unitk{2}{S})$ admits the same lower bound, since the~$T$-marginal of~$\unitk{2}{S}$ has spectrum~$(1, 0, \dotsc, 0)$ if~$S \subseteq T$ or~$S \subseteq \overline{T}$, and spectrum~$(1/2, 1/2, 0, \dotsc, 0)$ otherwise, with entropy~$0$ or~$1$ respectively.
\end{proof}
\begin{remark}
    We note that~\cref{lem:ncfunc on subset units} can fail if~$\theta \not\in \distsnc$ (due to the failure of~\cref{prop:E_upper_M_bound}), see~\cref{sec:crossing}.
\end{remark}
\begin{corollary}
    \label{cor:ncfunc on all but one}
    For~$\theta \in \distsnc$, $E^\theta(\unitk{2}{[k] \setminus \{i\}}) = 1 - \theta_{\{\{i\}, [k] \setminus \{i\}\}}$.
\end{corollary}

We use the lemma and the corollary to show that the laminar quantum functionals are distinct functions over $\Tnc^k$:
\begin{lemma}[Recognition principle]
    \label{lem:recognition}
    The mapping $\distsnc^k\to\R^{2^k-2}$ defined by $\theta\mapsto\big(E^\theta(\unitk{2}{S})\big)_{\emptyset\neq S\subsetneq[k]}$ is injective.
\end{lemma}
\begin{proof}
    First, one can learn the ``singleton weights'' on bipartitions~$\{\{i\}, [k] \setminus \{i\}\}$ using~\cref{cor:ncfunc on all but one}.
    For a bipartition~$\{\{i,j\}, [k] \setminus \{i,j\}\}$, one then has
    \[
        E^\theta(\unitk{2}{[k] \setminus \{i,j\}}) = 1 - \theta_{\{\{i,j\}, [k] \setminus \{i,j\}\}} - \theta_{\{\{i\}, [k] \setminus \{i\}\}} - \theta_{\{\{j\}, [k] \setminus \{j\}\}}
    \]
    and the latter two quantities are known.
    Similarly, assume one knows the weight of~$\theta$ on bipartitions where one side has cardinality~$\geq k-l+1$ with~$l \leq k/2$.
    Then for~$\emptyset \neq S \subsetneq [k]$ with cardinality~$k-l$, 
    if~$b = \{S, \overline{S}\}$, we have
    \[
        E^\theta(\unitk{2}{S}) = 1 - \theta_b - \sum_{\substack{b' = \{ T,  \overline{T} \} \\ S \subsetneq T}} \theta_{b'}  
    \]
    and the~$\theta_{b'}$ are all known, so one can learn~$\theta_b$.
\end{proof}

\subsection{Laminar functionals provide new asymptotic information}

In this subsection we show that the laminar functionals provide asymptotic information beyond what is accessible by the singleton(-after-grouping) quantum functionals, which are currently the only known universal spectral points.
The specific construction we use to prove this claim is inspired by the recent work~\cite{almanEdgeAsymptoticSpectrum2026}, see~\cref{remark:comparison to ALP26}.
For integers~$p, q \geq 1$, we consider the $4$-tensor
\[
  \psi_{p,q} = \unitk{p}{AC} \oplus \unitk{q}{AD} \oplus \unitk{q}{BC} \oplus \unitk{p}{BD}.
\]
Define
\[
  r_{p,q} = 2 \min_{t \in [0,1]} (p^t + q^{1-t}).
\]
The claim is the following:
\begin{theorem}
  \label{thm:laminar functionals give new information}
  The minimum of all laminar quantum functionals is $\min_{\theta \in \distsnc} F^\theta(\psi_{p,q}) = r_{p,q}$, and is in fact equal to the asymptotic subrank: $\asympQ(\psi_{p,q}) = r_{p,q}$.
  When~$p,q \geq 2$ and~$p \neq q$, there is an~$\epsilon > 0$ such that~$F^\theta(\psi_{p,q}) \geq r_{p,q} + \epsilon$ for all~$\theta$ which are singleton-supported or singleton-supported after grouping.
\end{theorem}

Recall that the asymptotic subrank of a tensor equals the minimum value over all universal spectral points evaluated on the tensor (in fact, it suffices to take all spectral points of any subsemiring containing it)~\cite[Thm.~3.8]{Strassen1988spectrum}. Therefore the laminar quantum functionals provide upper bounds on the asymptotic subrank of tensors in $\Tnc$.

To facilitate the calculations required to prove \cref{thm:laminar functionals give new information}, we organize the evaluations of all laminar quantum functionals into~\cref{table:psipq qfunc}.
\begin{table}[t]
  \begin{center}
  \begin{tabular}{c|c|c}
    Supported partitions & $F^\theta(\psi_{p,q})$ & $F^\theta(\unitk{r}{AB})$ \\
    \hline
$AB|CD,C,D$ & $p^{\theta_{AB} + \theta_C} + q^{\theta_{AB} + \theta_D} + q^{\theta_{AB} + \theta_C} + p^{\theta_{AB} + \theta_D}$ & $1$ \\
    $AC|BD,B,D$ & $1 + q^{\theta_{AC} + \theta_D} + q^{\theta_{AC} + \theta_B} + p^{\theta_{B} + \theta_D}$ & $r^{\theta_{AC} + \theta_B}$ \\
    $AD|BC,B,C$ & $p^{\theta_{AD} + \theta_C} + 1 + q^{\theta_B + \theta_C} + p^{\theta_{AD} + \theta_B}$ & $r^{\theta_{AB} + \theta_B}$ \\
    $A,BC|AD,D$ & $p^{\theta_{A} + \theta_{BC}} + q^{\theta_A + \theta_D} + 1 + p^{\theta_{BC} + \theta_D}$ & $r^{\theta_{A} + \theta_{BC}}$ \\
    $A,BD|AC,C$ & $p^{\theta_{A} + \theta_{C}} + q^{\theta_A + \theta_{BD}} + q^{\theta_{BD} + \theta_C} + 1$ & $r^{\theta_{A} + \theta_{BD}}$\\
    $A,B,CD|AB$ & $p^{\theta_{A} + \theta_{CD}} + q^{\theta_A + \theta_{CD}} + q^{\theta_B + \theta_{CD}} + p^{\theta_B + \theta_{CD}}$ & $r^{\theta_{A} + \theta_{B}}$\\
    $A,B,C,D$ & $p^{\theta_A + \theta_C} + q^{\theta_A + \theta_D} + q^{\theta_B + \theta_C} + p^{\theta_B + \theta_D}$ & $r^{\theta_{A} + \theta_{B}}$\\
    \hline
    $A,B,C,D, AB|CD$ & $p^{\theta_A + \theta_C + \theta_{AB|CD}} + q^{\theta_A + \theta_D + \theta_{AB|CD}} + q^{\theta_B + \theta_C + \theta_{AB|CD}} + p^{\theta_B + \theta_D + \theta_{AB|CD}}$ & $r^{\theta_{A} + \theta_{B}}$\\
    $A,B,C,D, AC|BD$ & $p^{\theta_A + \theta_C} + q^{\theta_A + \theta_D + \theta_{AC|BD}} + q^{\theta_B + \theta_C + \theta_{AC|BD}} + p^{\theta_B + \theta_D}$ & $r^{\theta_{A} + \theta_{B} + \theta_{AC|BD}}$ \\
    $A,B,C,D, AD|BC$ & $p^{\theta_A + \theta_C + \theta_{AD|BC}} + q^{\theta_A + \theta_D} + q^{\theta_B + \theta_C} + p^{\theta_B + \theta_D + \theta_{AD|BC}}$ & $r^{\theta_{A} + \theta_{B} + \theta_{AD|BC}}$
  \end{tabular}
  \end{center}
  \caption{Evaluations of the singleton and laminar quantum functionals on the family of tensors~$\psi_{p,q}$ as a function of $p$ and $q$ and the family of tensors~$\unitk{r}{AB}$ as a function of $r$. The first seven rows correspond to the singleton(-after-grouping) supports, whereas the last three rows are the ``properly'' laminar supports.}
  \label{table:psipq qfunc}
\end{table}
We first establish that the minimum over the laminar quantum functionals is~$r_{p,q}$:
\begin{lemma}
  The minimum of~$F^\theta(\psi_{p,q})$ over~$\theta \in \distsnc$ is at most~$r_{p,q}$.
\end{lemma}
\begin{proof}
  Suppose~$\theta$ is supported on the bipartitions~$A|BCD$, $B|ACD$, $AC|BD$.
  Then using \cref{table:psipq qfunc} and the arithmetic-geometric mean inequality, we obtain 
  \[
    F^\theta(\psi_{p,q}) = p^{\theta_A} + q^{\theta_A + \theta_{AC|BD}} + q^{\theta_B + \theta_{AC|BD}} + p^{\theta_B} \geq 2 (p^{\frac{\theta_A + \theta_B}{2}} + q^{\theta_{AC|BD} + \frac{\theta_A + \theta_B}{2}}),
  \]
  with equality whenever~$\theta_A=\theta_B$.
  Minimizing over~$\theta$ with this support yields~$2 \min_{t \in [0,1/2]} (p^t + q^{1-t})$, which is~$r_{p,q}$ whenever~$p \geq q$.
  In case~$p \leq q$, $r_{p,q}$ can be achieved as an upper bound on~$\min_{\theta \in \distsnc} F^\theta(\psi_{p,q})$ by considering laminar quantum functionals supported on~$A|BCD$, $B|ACD$ and~$AD|BC$.
\end{proof}
Next, we lower bound the singleton(-after-grouping) quantum functionals:
\begin{lemma}
  For~$\theta$ singleton-supported on the following groupings, $F^\theta(\psi_{p,q})$ is lower bounded by:
  \begin{enumerate}
    \item\label{enum:Ftheta psipq a b c d} $A,B,C,D$:  $2 (\sqrt{p} + \sqrt{q})$.
    \item $AB,C,D$: $2 (\sqrt{p} + \sqrt{q})$.
    \item\label{enum:Ftheta psipq ac b d} $AC,B,D$: $1 + \min_{t \in [0,\frac{1}{2}]} (2 q^{1-t} + p^{2t})$.
    \item $AD,B,C$: $1 + \min_{t \in [\frac{1}{2},1]} (2 p^{t} + q^{2(1-t)})$.
    \item $A,BC,D$: $1 + \min_{t \in [\frac{1}{2},1]} (2 p^{t} + q^{2(1-t)})$.
    \item $A,BD,C$: $1 + \min_{t \in [0,\frac{1}{2}]} (2 q^{1-t} + p^{2t})$.
    \item $A,B,CD$: $2 (\sqrt{p} + \sqrt{q})$.
  \end{enumerate}
  When~$p, q \geq 2$ and~$p \neq q$, there exists~$\epsilon > 0$ such that each of the lower bounds above is at least~$r_{p,q} + \epsilon$.
\end{lemma}
\begin{proof}
  We only provide explicit proofs for~\ref{enum:Ftheta psipq a b c d} and~\ref{enum:Ftheta psipq ac b d} as the rest of the lower bounds follow from analogous arguments.
  To prove~\ref{enum:Ftheta psipq a b c d}, note that (cf.~\cref{table:psipq qfunc})
  \[
    F^\theta(\psi_{p,q}) = p^{\theta_A + \theta_C} + q^{\theta_A + \theta_D} + q^{\theta_B + \theta_C} + p^{\theta_B + \theta_D}
    \geq 2 (p^{(\theta_A + \theta_C + \theta_B + \theta_D)/2} + q^{(\theta_A + \theta_D + \theta_B + \theta_C)/2}) = 2 (\sqrt{p} + \sqrt{q})
  \]
  using the arithmetic-geometric mean inequality, and equality is achieved for $\theta_A=\theta_B=\theta_C=\theta_D=\frac{1}{4}$.
  Similarly for the $AC$, $B$, $D$ grouping, we have
  \[
    F^\theta(\psi_{p,q}) = 1 + q^{\theta_{AC} + \theta_D} + q^{\theta_{AC} + \theta_B} + p^{\theta_{B} + \theta_D}
    \geq 1 + 2 q^{\theta_{AC} + \frac{\theta_D + \theta_B}{2}} + p^{\theta_{B} + \theta_{D}}
  \]
  which is of the form~$1 + 2q^{1-t} + p^{2t}$ for~$t = \frac{\theta_B + \theta_D}{2} \in [0,\frac{1}{2}]$. Equality will hold for~$\theta_B=\theta_D$.

  The fact that all of these are strictly larger than~$r_{p,q}$ whenever~$p,q \geq 2$ and~$p \neq q$, follows from the arithmetic-geometric mean inequality being strict whenever the two terms are distinct. For example, in the case of $AC$, $B$, $D$, we have
  \[
    1 + 2q^{1-t} + p^{2t} > 2(q^{1-t} + p^{t}) \geq r_{p,q}.
  \]
  Since the quantum functionals depend continuously on~$\theta$, and the set of considered~$\theta$ is compact, we conclude that there exists~$\epsilon > 0$ such that~$F^\theta(\psi_{p,q}) \geq r_{p,q} + \epsilon$ for all~$\theta$ which are singleton-supported (after grouping).
\end{proof}
Finally, let us establish the lower bound on the asymptotic subrank.
We use the following reformulation of~$r_{p,q}$.
\begin{lemma}
  For~$p,q \geq 1$, we have the identity
  \[
    r_{p,q} = 2 \min_{t \in [0,1]} (p^t + q^{1-t})
    =
    \max_{s \in [0,1]} 2^{h(s)+1} \min(p^s, q^{1-s}),
  \]
  where~$h$ is the binary entropy defined by
  \[
    h(s) = - s \log_2(s) - (1-s) \log_2(1-s).
  \]
\end{lemma}
\begin{proof}
  Let~$P = \log p$, $Q = \log q$, so that $\min(p^s, q^{1-s}) = e^{\min(s P, (1-s) Q)} = \min_{t \in [0,1]} e^{t s P + (1-t) (1-s) Q}$.
  Write~$h_e(s) = - s \log(s) - (1-s) \log(1-s)$ for the entropy measured in nats.
  Then
  \begin{align*}
    \max_{s \in [0,1]} (\min_{t \in [0,1]} h_e(s) + t s P + (1-t) (1-s) Q)
    & = \min_{t \in [0,1]} (\max_{s \in [0,1]} h_e(s) + t s P + (1-t) (1-s) Q) \\
    & = \min_{t \in [0,1]} \log(e^{t P} + e^{(1-t) Q}),
  \end{align*}
  where the first equality follows from von Neumann's minimax theorem, and the second equality is a standard identity (which may be established by differentiating with respect to $s$ and finding the root at~$s = e^{tP}/(e^{tP} + e^{(1-t)Q}) \in (0,1)$).
  Exponentiating yields the lemma.
\end{proof}
\begin{lemma}[{\cite[Sec.~5.2]{christandl2020tensornetworksreps}}]
  \label{lem:4cycle border subrank}
  For~$d \geq 1$, the border subrank of~$\unitk{d}{AC} \kron \unitk{d}{AD} \kron \unitk{d}{BC} \kron \unitk{d}{BD}$ is at least~$\lceil \frac{d^2}{2} \rceil$.
\end{lemma}
\begin{corollary}
  The asymptotic subrank of~$\asympQ(\psi_{p,q})$ satisfies
  \[
    \asympQ(\psi_{p,q}) \geq \max_{s \in [0,1]} 2^{h(s) + 1} \min(p^s, q^{s-1}).
  \]
\end{corollary}
\begin{proof}
  Let us assume for a moment that~$s \in [0,1]$ is rational and take $n$ so that both~$sn/2$ and~$(1-s)n/2$ are integers. Then
  \begin{align*}
    \psi_{p,q}^{\kron n} 
    & \cong \bigoplus_{n_{AC} + n_{AD} + n_{BC} + n_{BD} = n} \binom{n}{n_{AC}, n_{AD}, n_{BC}, n_{BD}} 
  \unitk{p^{n_{AC}}}{AC} \kron \unitk{q^{n_{AD}}}{AD} \kron \unitk{q^{n_{BC}}}{BC} \kron \unitk{p^{n_{BD}}}{BD} \\
    & \geq \binom{n}{\frac{s n}{2}, \frac{(1-s) n}{2}, \frac{(1-s) n}{2}, \frac{s n}{2}} 
    \unitk{p^{s n/2}}{AC} \kron
    \unitk{q^{(1-s) n/2}}{AD} \kron
    \unitk{q^{(1-s) n/2}}{BC} \kron
    \unitk{p^{s n/2}}{BD} \\
    & \geq \binom{n}{\frac{s n}{2}, \frac{(1-s) n}{2}, \frac{(1-s) n}{2}, \frac{s n}{2}} 
    \unitk{m}{AC} \kron
    \unitk{m}{AD} \kron
    \unitk{m}{BC} \kron
    \unitk{m}{BD}, \quad m_n := \min(p^{sn/2}, q^{(1-s)n/2}) \\
    & \degengeq \binom{n}{\frac{s n}{2}, \frac{(1-s) n}{2}, \frac{(1-s) n}{2}, \frac{s n}{2}} \unitk{\ceil{m_n^2/2}}{ABCD}
  \end{align*}
  where the last inequality follows from~\cref{lem:4cycle border subrank}.
  From this we observe that the border subrank satisfies
  \begin{align*}
    \bordersubrank(\psi_{p,q}^{\kron n})^{1/n}
      &\geq
      \left[\binom{n}{\frac{s n}{2}, \frac{(1-s) n}{2}, \frac{(1-s) n}{2}, \frac{s n}{2}} \ceil{m_n^2/2}\right]^{\frac{1}{n}} \\
      &= \binom{n}{\frac{s n}{2}, \frac{(1-s) n}{2}, \frac{(1-s) n}{2}, \frac{s n}{2}}^{\frac{1}{n}} \left(\ceil{\min(p^{sn}, q^{(1-s)n})/2}\right)^{\frac{1}{n}} \\
      &\stackrel{n\to\infty}{\longrightarrow} 2^{h(s) + 1} \min(p^{s}, q^{1-s})
  \end{align*}
  As the lower bound is continuous in~$s$, we conclude that it holds for arbitrary~$s \in [0,1]$.
\end{proof}

\begin{remark}
  \label{remark:comparison to ALP26}
  In~\cite{almanEdgeAsymptoticSpectrum2026} it is shown for~$\psi_{p,q}$ that the \emph{asymptotic multilinear commutative rank} $\asympCR_{AB}(\psi_{p,q})$ is exactly~$r_{p,q}$.
  The multilinear commutative rank~$\CR_{AB}(\psi)$ (over~$\C$) is equal to the maximal~$r$ such that~$\psi \geq \unitk{r}{AB}$, and the asymptotic version is its regularization.
  The lower bound follows from a simple construction (along the same lines of the above construction for the asymptotic subrank), but the upper bound requires a more delicate argument.
  The laminar quantum functionals are capable of establishing this upper bound as well: indeed when~$\theta$ is a weighting supported on~$A$, $B$ and~$AC|BD$, then
  $F^\theta(\unitk{r}{AB}) = r$, and
  \[
    F^\theta(\psi_{p,q}) = p^{\theta_A} + q^{\theta_A + \theta_{AC|BD}} + q^{\theta_B + \theta_{AC|BD}} + p^{\theta_B} \geq 2 (p^{\frac{\theta_A + \theta_B}{2}} + q^{\theta_{AC|BD} + \frac{\theta_A + \theta_B}{2}}).
  \]
  Minimizing over~$\theta$ with this support yields~$2 \min_{t \in [0,1/2]} (p^t + q^{1-t})$, which is~$r_{p,q}$ whenever~$p \geq q$, and therefore~$\asympCR_{AB}(\psi_{p,q}) \leq r_{p,q}$.
  In case~$p \leq q$, the same upper bound can be achieved by considering laminar quantum functionals supported on~$A$, $B$ and~$AD|BC$.

  By the same argument as for the asymptotic subrank, one can argue that this upper bound on~$\asympCR_{AB}(\psi_{p,q})$ is not provable using the singleton(-after-grouping) quantum functionals. 
  In fact, for any proper flattening of~$\psi_{p,q}$, the minimum of the quantum functionals is equal to the asymptotic subrank of that flattening, since all such flattenings are tight. 
  This argument is also used in~\cite{almanEdgeAsymptoticSpectrum2026} to reason that more spectral points must exist, as they show that~$\asympCR$ is the pointwise minimum of some family of universal spectral points.
\end{remark}

\subsection{Laminar functionals agree on \texorpdfstring{$\W_4$}{W4}}\label{sec:agree_on_W}
We will now show that the upper and lower quantum functionals which are parameterized by the same laminarly-supported distribution agree on $\W_4 \in \C^2 \ot \C^2 \ot \C^2 \ot \C^2$, defined by
\[
\W_4 \coloneq e_2 \otimes e_1 \otimes e_1 \otimes e_1 + e_1 \otimes e_2 \otimes e_1 \otimes e_1 + e_1 \otimes e_1 \otimes e_2 \otimes e_1 + e_1 \otimes e_1 \otimes e_1 \otimes e_2.
\]
\vspace{-2em}
\begin{theorem}\label{thm:W4-non-crossing-agree}
        For any $\theta=(\theta_{AB},\theta_A,\theta_B,\theta_C,\theta_D)$, $E^\theta(\W_4)= E_\theta(\W_4)$.
\end{theorem}
By symmetry of~$\W_4$, we deduce that the above holds for any~$\theta \in \distsnc^4$, and hence:
\begin{corollary}
    $\W_4 \in \Tnc^4$.
\end{corollary}
\begin{remark}
    Note that~$\W_4$ is not in the semiring generated by embedded~$3$-tensors, $\T_3^4$.
    First, observe that anything in~$\T_3^4$ is a direct sum of Kronecker products of embedded $3$-tensors, using distributivity. 
    If $\W_4$ is a direct sum of non-trivial 4-tensors, it can only be the direct sum of two~$1 \times 1 \times 1 \times 1$ tensors. Such a construction produces a semistable tensor (in fact, it is equivalent to~$\unitk{2}{4}$), but~$\W_4$ is unstable, hence not a direct sum of~$1 \times 1 \times 1 \times 1$ tensors.
    The only remaining option is for $\W_4$ to be a Kronecker product of two tensors. By multiplicativity of flattening ranks, the only option therefore is for $\W_4$ to be the Kronecker product of tensors of format~$2 \times 2 \times 1 \times 1$ and~$1 \times 1 \times 2 \times 2$ (up to permutation of the parties). 
    However, this would imply flattening rank of~$1$ across this bipartition, whereas all flattening ranks of~$\W_4$ are~$2$.
\end{remark}
To prove \cref{thm:W4-non-crossing-agree} we will use the following claim:
\begin{claim}\label{cl:ineqs-W4-upper-quantum-func}
    Let $\alpha,\beta,\gamma,\delta,z$ be the normalized second-rows of integer partitions of $n$ corresponding to $A,B,C,D, AB$ respectively, such that the composition of their Young projectors does not vanish on $\W_4$. Then
    \begin{enumerate}
        \item\label{eq:trivialBoundsSecondRows} $0 \leq \alpha, \beta, \gamma, \delta, z \leq \frac12$.
        \item\label{eq:triangle-inequality} $\abs{\alpha-\beta} \leq z$, and $\abs{\gamma-\delta} \leq z$.
        \item\label{eq:firstRowBoundingSecondRows} $\alpha + \beta \leq 1-z$, and $\gamma+\delta \leq 1-z$.
        \item\label{eq:W4-moment-polytope} $\alpha + \beta + \gamma + \delta \leq 1$.
    \end{enumerate}
\end{claim}
\begin{proof}[Proof of \Cref{thm:W4-non-crossing-agree}]
        Any distribution $\theta$ has $E^\theta\geq E_\theta$ \cite[Theorem 3.24 (in the arXiv version)]{cvz23universal}. Here we prove the inequality $E^\theta(\W_4)\leq E_\theta(\W_4)$.
        
        Define $H_\theta(\alpha,\beta,\gamma,\delta,z)\coloneq\theta_{AB}h(z)+\theta_Ah(\alpha)+\theta_Bh(\beta)+\theta_Ch(\gamma)+\theta_Dh(\delta)$ the weighted average of the entropies for the input, according to $\theta$. The entropies $h$ are binary entropies, as we deal with Schmidt-rank $2$ for every bipartition (e.g. for $\W_4$ itself see the proof of \cref{cl:ineqs-W4-upper-quantum-func}). This is the optimization function by which both $E_\theta$ and $E^\theta$ are defined, over their respective feasible, or allowed, $(\alpha,\beta,\gamma,\delta,z)$'s.
        
        Let $x\coloneq(\alpha,\beta,\gamma,\delta,z)$ be the normalized second-rows of partitions for which $P^{AB|CD}_zP^{A|B|C|D}_{\alpha,\beta,\gamma,\delta}\W_4^{\ot n}\neq0$, i.e. a feasible point for $E^\theta(\W_4)$.
        We will show there is another feasible point for $E_\theta(\W_4)$, $x'\coloneq(\alpha',\beta',\gamma',\delta',z')$, which satisfies $H_\theta(x')\geq H_\theta(x)$. 
        Note $x'$ is feasible if it is the vector of second-eigenvalues of the corresponding reduced density matrices of a normalized state that was SLOCC-converted (i.e. restricted) from $\W_4$. For any $\alpha',\beta',\gamma',\delta'$ which sum to $1$ we have the following such state $A_1\otimes A_2\otimes A_3\otimes A_4 (\W_4)=\sqrt{\alpha'}e_2\otimes e_1\otimes e_1\otimes e_1+\sqrt{\beta'}e_1\otimes e_2\otimes e_1\otimes e_1+\sqrt{\gamma'}e_1\otimes e_1\otimes e_2\otimes e_1+\sqrt{\delta'}e_1\otimes e_1\otimes e_1\otimes e_2$ (e.g. $A_1=\begin{pmatrix}
            1&0\\0&\sqrt{\alpha'}
        \end{pmatrix}$). We will show we can define these, together with $z'$, in a consistent manner (either $z'=\alpha'+\beta'$ or $z'=\gamma'+\delta'$, by $h(z')=h(1-z')$) which also achieves a weighted entropy at least as large as that of $x$, which will finish the proof.

        By \cref{cl:ineqs-W4-upper-quantum-func} we know $0\leq\alpha,\beta,\gamma,\delta,z\leq\tfrac12$, $\alpha+\beta,\gamma+\delta\leq1-z$, and $\sigma\coloneq\alpha+\beta+\gamma+\delta\leq1$.
        Split between the equality and the strict inequality cases, to define $x'$: 
        
        If $\sigma=1$, we keep $\alpha'\coloneq\alpha,\ \beta'\coloneq\beta,\ \gamma'\coloneq\gamma,\ \delta'\coloneq\delta$ and only (possibly) change $z$. $1=(\alpha+\beta)+(\gamma+\delta)\leq1-z+(\gamma+\delta)$ gives $z\leq\gamma+\delta$, and similarly $z\leq\alpha+\beta$. Define $z'\coloneq\min\{\alpha+\beta,\gamma+\delta\}$, as this is the marginal entropy's second value for the state defined by $\alpha,\beta,\gamma,\delta$. As both possible values are at least as large as $z$, $z\leq z'$. Together the two pairs add up to $1$, making $z\leq z'\leq\tfrac12$ which implies $h(z')\geq h(z)$. The other entropies didn't change so we're done. 
        
        If, however, there is a strict inequality $\sigma<1$, we will define $x'$ that will "correct" this into an equality, without decreasing the entropy. First we will correct the sum by adding to $\alpha,\gamma$, making the smaller pair larger, and if need be enlarging the larger pair as well. The strict inequality implies $\min\{\alpha+\beta,\gamma+\delta\}<\tfrac12$. Assume $\gamma+\delta\leq\alpha+\beta$ (otherwise swap between them in the following) and define $\gamma'\coloneq\gamma+\min\{\tfrac12-\gamma-\delta,1-\sigma\}$ and $\alpha'\coloneq\alpha+\max\{0,\tfrac12-\alpha-\beta\}$. 
        Note that both $\gamma'$ and $\alpha'$ are closer to $\tfrac12$ than $\gamma$ and $\alpha$, so $h(\gamma')\geq h(\gamma)$ and $h(\alpha')\geq h(\alpha)$. We made $\sigma'\coloneq\alpha'+\beta+\gamma'+\delta=1$ while keeping $\alpha',\beta,\gamma',\delta\leq\tfrac12$ and $1-z\geq\alpha'+\beta\geq\gamma'+\delta$ ($\alpha'>\alpha$ only if $\alpha+\beta<\tfrac12$ and in this case $\alpha'+\beta=\tfrac12$). Thus, $(\alpha',\beta,\gamma',\delta,z)$ satisfy the inequalities used for the $\sigma=1$ case which we can repeat now: $1=(\alpha'+\beta)+(\gamma'+\delta)\leq1-z+(\gamma'+\delta)$ implying $z\leq\gamma'+\delta\leq\alpha'+\beta$. Define $z'\coloneq\gamma'+\delta$. We have $z\leq\gamma'+\delta\leq \tfrac12$ meaning $h(z')\geq h(z)$. Set $\beta'\coloneq\beta,\ \delta'\coloneq\delta$ and we're done. 
        
        In both cases we have $H_\theta(x)\leq H_\theta(x')$ while $x'$ is a feasible vector for $E_\theta(\W_4)$
        as required.
\end{proof}

To establish~\cref{cl:ineqs-W4-upper-quantum-func} we use the following:
\begin{theorem}[Thm 1.5(b) in \cite{klemm77}, Thm 1.6(a) in \cite{DVIR1993125}, \cite{clausen1993}]\label{thm:klemmDvirClausenMaier}
        For all partitions $\mu,\nu,\lambda$ with $g_{\mu\nu\lambda}\neq0$, $\nu_1\leq|\lambda\cap\mu|=\sum_i\min\{\lambda_i,\mu_i\}$, and there is such a triple with equality.
    \end{theorem}

\begin{lemma}[Complementing projections agree on symmetric states, lemma 3.1 in (arXiv ver. of) \cite{cvz23universal}] \label{lem:complementaryProjectionsOnSymmetric}
    Let $S\subsetneq[k]$ be some proper non-empty subset of the $k$-parties of a tensor in $V_1\otimes\cdots\otimes V_k$, and $\lambda\vdash n$ a partition. Let $P^{V_{[k]}}_{(n)}$ be the projection of $(V_1\otimes\cdots\otimes V_k)^{\otimes n}$ to the symmetric subspace, and $P^{V_S}_\lambda,P^{V_{\overline{S}}}_\lambda$ the projections to the $\lambda$-isotypical subspaces in the parties that are in $S$ or $\overline{S}$ respectively. Then
    \begin{enumerate}
        \item $P^{V_S}_\lambda P^{V_{[k]}}_{(n)}=P^{V_{[k]}}_{(n)}P^{V_S}_\lambda$ and
        \item $P^{V_S}_\lambda P^{V_{[k]}}_{(n)}=P^{V_{\overline{S}}}_\lambda P^{V_{[k]}}_{(n)}$
    \end{enumerate}
\end{lemma}

\begin{proof}[Proof of \Cref{cl:ineqs-W4-upper-quantum-func}]
    Let $\mu_A,\mu_B,\mu_C,\mu_D, \mu_{AB}$ be the corresponding integer partitions of $n$ such that $P^{V_A}_{\mu_A}P^{V_B}_{\mu_B}P^{V_C}_{\mu_C}P^{V_D}_{\mu_D}P^{V_A\otimes V_B}_{\mu_{AB}}(\W_4)\neq0$. 
    
    \ref{eq:trivialBoundsSecondRows} As the local dimension of each of the four parties in $\W_4$ is $2$ (qubits), their corresponding partitions are of at most two rows. $\mu_{AB}$ is also of at most two rows, as $\W_4=(e_{2} \otimes e_1 + e_1 \otimes e_2) \otimes e_1\otimes e_1 + e_1\otimes e_1 \otimes (e_2 \otimes e_1 + e_1 \otimes e_2)$. Denote the second coefficients of the normalised partitions $\bar{\mu}_A(=\tfrac1{n}\mu_A),\ldots,\bar{\mu}_D,\bar{\mu}_{AB}$ by $\alpha,\beta,\gamma,\delta,z$ respectively. As the normalisation of the second rows, meaning the shorter ones, these are all at most $\tfrac12$.

    \ref{eq:triangle-inequality} For the inequalities~$\abs{\alpha-\beta} \leq z$, $\abs{\gamma-\delta} \leq z$ we will use \cref{thm:klemmDvirClausenMaier}.
    Applying it to $\nu\coloneq\mu_{AB},\mu\coloneq\mu_A,\lambda\coloneq\mu_B$, as $\W_4$ is non-vanishing for the composition of their projectors, the Kronecker coefficient is non-zero. By the theorem this implies $1-z\leq \min\{1-\alpha,1-\beta\}+\min\{\alpha,\beta\}=\min\{1-\alpha+\beta,1-\beta+\alpha\}=1-|\alpha-\beta|$ which is $|\alpha-\beta|\leq z$. For $\abs{\gamma-\delta}\leq z$ we do the same with $\nu\coloneq\mu_{AB},\mu\coloneq\mu_C,\lambda\coloneq\mu_D$. To use $\mu_{AB}$ on the pair $CD$ we use \cref{lem:complementaryProjectionsOnSymmetric} and the fact that $\W_4$ is a symmetric tensor, implying that projecting $AB$ on $\mu_{AB}$ is the same as projecting $CD$ on the same isotypic component $\mu_{AB}$.

\ref{eq:firstRowBoundingSecondRows} For the inequalities~$\alpha + \beta \leq 1-z$, $\gamma+\delta \leq 1-z$ we observe that from the perspective of the tripartition $\{AB\},\{C\},\{D\}$, $\W_4$ is equivalent to the tripartite $\W_3$ (as shown above).
    Using the defining inequality for the entanglement polytope of $\W_3$, we know the three normalized first-rows of the single-parties' partitions must add up to at least $2$ (see~\cite{entanglementPolytope2013}): $(1-\gamma)+(1-\delta)+(1-z)\geq2$ which translates to $\gamma+\delta\leq1-z$. Doing the same for $\{A\},\{B\},\{CD\}$ gives $\alpha+\beta\leq1-z$.

    \ref{eq:W4-moment-polytope} The inequality~$\alpha + \beta + \gamma + \delta \leq 1$ follows directly from from~\cite[Thm.~6.1]{nessStratificationNullCone1984}; alternatively, the moment (entanglement) polytope of~$\W_4$ is given explicitly in~\cite[p.1207 or eq. S17 in Supplemental Materials]{entanglementPolytope2013}.
\end{proof}

\begin{remark}
    The inequalities in \cref{cl:ineqs-W4-upper-quantum-func}\ref{eq:triangle-inequality} generalize to what are called \emph{polygonal} inequalities in the quantum information theory community~\cite{higuchiOneQubitReducedStates2003,bravyiCompatibilityLocalMultipartite2004}.
    For instance, for general~$4$-qubit states, one has~$\alpha+\beta+\gamma \geq \delta$. In our case this inequality is implied by~$\alpha+\beta \leq 1-z \leq 1-\gamma+\delta$, but the inequality~$\alpha+\beta \leq 1-z$ is specific for~$\W_4$ (or rather, it comes from~$\W_3$'s entanglement polytope).
\end{remark}

\ifanonymous
\else
\section*{Acknowledgements}
AB acknowledges support from Uniandes Faculty of Science project No. INV-2023-162-2814.
MC, TF and HN acknowledge financial support from Villum Fonden via the QMATH Centre of Excellence (Grant No.~10059) and the Novo Nordisk Foundation (grant NNF20OC0059939 `Quantum for Life').
HN also acknowledges support by the European Union via a ERC grant (QInteract, Grant No.~101078107), and via HORIZON-MSCA-PF-2024-PF-01 grant agreement, project number: 101212204 (AsympTensorPolytope).
IL thanks the QMATH Center at the University of Copenhagen for the hospitality during his stays, and the Danish-Israeli Study Foundation in Memory of Josef and Regine Nachemson for funding his visits. The research of IL was funded by the European Union (ERC, EACTP, 101142020).
Views and opinions expressed are however those of the author(s) only and do not necessarily reflect those of the European Union or the European Research Council Executive Agency. Neither the European Union nor the granting authority can be held responsible for them.
\fi

\appendix

\section{Maximizing entropy with determinant constraints}
\label{sec:max_ent_with_det}
In this section, we determine the maximum possible entropy of a probability distribution $p$ over $\{1,\ldots, k\}$ subject to the determinant-like constraint $p_1 p_2 \cdots p_k \leq c$, which we denote by $S_k(c)$.

\begin{proposition}
\label{prop:Edetab-exact-value}
    Let $c > 0$ and $k \geq 2$ and let $S_k(c)$ be defined as follows:
    \begin{equation}
        S_k(c) \coloneq \max\left\{ - \sum_{i=1}^{k}p_i\log_2 p_i \,\middle|\, p_1 \geq p_2 \geq \cdots \geq p_k \geq 0,  \sum_{i=1}^{k}p_i = 1, \prod_{i=1}^{k} p_i \leq c \right\}.
    \end{equation}
    \begin{itemize}
        \item If $c \geq (\frac{1}{k})^{k}$, then $S_k(c) = \log_2 k$ with maximum achieved by uniform $p_1 = \cdots=p_k=\frac{1}{k}$.
        \item Otherwise if $c < (\frac{1}{k})^{k}$, then
        \begin{align}
          \label{eq:Skc value}
             S_k(c) 
             &= -\gamma \log_2\frac{\gamma}{k-1} - (1-\gamma) \log_2 (1-\gamma).
        \end{align}
        with maximum achieved by $p_1=p_2=\cdots=p_{k-1}=\frac{\gamma}{k-1}$ and $p_k = 1- \gamma$
        where $\gamma \in (1-\frac{1}{k},1)$ is the unique solution to the determinant constraint $\left(\frac{\gamma}{k-1}\right)^{k-1}(1-\gamma) = c$.
    \end{itemize}
\end{proposition}
\begin{proof}
    First observe that if $c \geq (\frac{1}{k})^{k}$, then
    $S_k(c) = \log_2 k$ and is achieved when the distribution is uniform,
    i.e., $p_i=\frac{1}{k}$ for all $i$.
    Henceforth we only consider the case where
    $c < (\frac{1}{k})^{k}$.
    Moreover, without loss of generality, we may assume saturation,
    $\prod_{i=1}^{k} p_i = c$, holds at the optimal value\footnote{Otherwise, if a distribution $q$ maximizes the entropy
    $-\sum_i q_i\log_2 q_i$ while $\prod_{i=1}^{k}q_i<c$, it would
    constitute a local optimum distinct from the uniform distribution
    for the unconstrained optimization, contradicting the strict
    concavity and uniqueness of the optimum for the unconstrained
    optimization.}.
    The Lagrangian for the resulting optimization is
    \begin{equation*}
        \mathcal L(p,\mu_1,\mu_2)
        \coloneq
        -\sum_{i=1}^{k}p_i\log_2p_i
        +\mu_1\left(\sum_{i=1}^{k}p_i-1\right)
        +\mu_2\left(\sum_{i=1}^{k}\log_2p_i-\log_2c\right).
    \end{equation*}
    Therefore the critical values for $(p_1,p_2,\ldots,p_k)$ must
    satisfy, for each $i\in\{1,\ldots,k\}$, the equation
    \begin{equation*}
        0
        =\partial_{p_i}\mathcal L
        =-\log_2p_i-\alpha+\mu_1+\frac{\alpha\mu_2}{p_i},
    \end{equation*}
    where we have introduced
    \[
        \alpha \coloneqq \log_2 e
    \]
    to facilitate working with logarithms in base $2$.
    Equivalently, the critical values for $(p_1,p_2,\ldots,p_k)$ must
    satisfy
    \begin{equation*}
        p_i\log_2p_i=(\mu_1-\alpha)p_i+\alpha\mu_2.
    \end{equation*}
    Since the left-hand side is a convex function in $p_i$ and the
    right-hand side is linear, there are at most two possible values
    for $p_i$.
    Therefore, the critical values for $(p_1,p_2,\ldots,p_k)$ are
    bimodal, meaning they are of the form
    \begin{equation}
        (p_1,p_2,\ldots,p_k)
        =
        \left(
        \underbrace{\frac{\lambda}{m},\ldots,\frac{\lambda}{m}}_
                    {m\text{ times}},
        \underbrace{\frac{1-\lambda}{k-m},\ldots,
                    \frac{1-\lambda}{k-m}}_{k-m\text{ times}}
        \right)
        \label{eq:bimodal_dist}
    \end{equation}
    for some $m\in\{1,2,\ldots,k-1\}$ and $\lambda\in[0,1]$
    satisfying the determinant constraint
    \begin{equation*}
        F(m,\lambda)
        \coloneqq f_m(\lambda)
        =
        c-
        \left(\frac{\lambda}{m}\right)^m
        \left(\frac{1-\lambda}{k-m}\right)^{k-m}
        =0.
    \end{equation*}
    By sorting the probabilities $(p_1,p_2,\ldots,p_k)$ and relabelling
    $m\leftrightarrow k-m$, we may also assume that
    \[
        \frac{\lambda}{m}
        \geq
        \frac{1-\lambda}{k-m},
    \]
    or equivalently
    \begin{equation*}
        \lambda\geq\frac{m}{k}.
    \end{equation*}
    For the sake of subsequent analysis, now suppose $m$ need not be an integer
    and $m\in(0,k)$. Since
    \[
        f_m(0)=f_m(1)=c>0,
        \qquad
        f_m\left(\frac{m}{k}\right)
        =
        c-\left(\frac{1}{k}\right)^k<0,
    \]
    the intermediate value theorem implies there always exists a real
    root of $f_m$ for $\lambda\in(0,\frac{m}{k})$ and another real root
    of $f_m$ for $\lambda\in(\frac{m}{k},1)$.
    Moreover, as $f_m(\lambda)\leq c$ for all $\lambda\in[0,1]$, and
    \begin{align*}
        f_m'(\lambda)
        &=
        \frac{k(c-f_m(\lambda))}{\lambda(1-\lambda)}
        \left(\lambda-\frac{m}{k}\right),
    \end{align*}
    we conclude that $f_m'(\lambda)>0$ for all
    $\lambda\in(\frac{m}{k},1)$. Thus $f_m$ is strictly increasing on
    $(\frac{m}{k},1)$, and the real root in this interval is both simple
    and unique.
    Let $q(m)\in(\frac{m}{k},1)$ be this unique real root of $f_m$,
    meaning
    \[
        F(m,q(m))=f_m(q(m))=0.
    \]
    To summarize, we have established that the determinant-constrained
    maximum entropy $S_k(c)$ is given by
    \begin{equation*}
        S_k(c)
        =
        \max_{m\in\{1,2,\ldots,k-1\}}h(m),
    \end{equation*}
    where $h(m)$ is the entropy of the distribution in
    \cref{eq:bimodal_dist}, with $\lambda=q(m)$:
    \begin{equation*}
        h(m)
        \coloneqq
        -q(m)\log_2\frac{q(m)}{m}
        -(1-q(m))\log_2\frac{1-q(m)}{k-m}.
    \end{equation*}

    To conclude, we will establish that $m=k-1$ is the maximizer for
    $S_k(c)$, i.e., $S_k(c)=h(k-1)$, by proving that $h(m)$ is an
    increasing function on the interval $m\in(0,k)$.
    By differentiating $h(m)$ we observe
    \begin{equation*}
        h'(m)
        =
        q'(m)
        \left(
            \log_2\frac{1-q(m)}{k-m}
            -
            \log_2\frac{q(m)}{m}
        \right)
        +
        \frac{\alpha k}{m(k-m)}
        \left(q(m)-\frac{m}{k}\right).
    \end{equation*}
    From here we can compute $q'(m)$ because
    \begin{equation*}
        0=F(m,q(m))
        \implies
        0
        =
        \frac{\mathrm{d}[F(m,q(m))]}{\mathrm{d}m}
        =
        (\partial_\lambda F)(m,q(m))q'(m)
        +(\partial_mF)(m,q(m)),
    \end{equation*}
    and therefore
    \begin{equation*}
        q'(m)
        =
        -\frac{(\partial_mF)(m,q(m))}
              {(\partial_\lambda F)(m,q(m))}.
    \end{equation*}
    Note that
    $(\partial_\lambda F)(m,q(m))=f_m'(q(m))>0$, so this $q'(m)$ is
    well-defined. In fact,
    \begin{equation*}
        (\partial_\lambda F)(m,q(m))
        =
        f_m'(q(m))
        =
        \frac{kc}{q(m)(1-q(m))}
        \left(q(m)-\frac{m}{k}\right)>0,
    \end{equation*}
    and
    \begin{equation*}
        (\partial_mF)(m,\lambda)
        =
        \frac{c-f_m(\lambda)}{\alpha}
        \left(
            \log_2\frac{1-\lambda}{k-m}
            -
            \log_2\frac{\lambda}{m}
        \right).
    \end{equation*}
    Thus, for $\lambda=q(m)$,
    \begin{equation*}
        (\partial_mF)(m,q(m))
        =
        \frac{c}{\alpha}
        \left(
            \log_2\frac{1-q(m)}{k-m}
            -
            \log_2\frac{q(m)}{m}
        \right).
    \end{equation*}
    Let
    \[
        \delta\coloneqq q(m)-\frac{m}{k}>0
    \]
    and also
    \[
        L
        \coloneqq
        \log_2
        \left(
            \frac{q(m)(k-m)}{m(1-q(m))}
        \right)>0
    \]
    where $L > 0$ holds because $q(m) > m/k$.
    Altogether, we obtain the expression
    \begin{equation*}
        h'(m)
        =
        \frac{\alpha k\delta}{m(k-m)}\left[
            1
            -
            \frac{m(k-m)q(m)(1-q(m))}{(k\delta)^2}
            \left(\frac{L}{\alpha}\right)^2
        \right].
    \end{equation*}
    Finally, to prove $h'(m) > 0$ we eliminate the $\log_2$ from $L$ by using the inequality
    \[
        r > 1 \implies \log_2 r < \alpha \frac{r-1}{\sqrt{r}}
    \]
    Therefore,
    \begin{equation*}
        \frac{L}{\alpha }
        <
        \frac{\frac{q(m)(k-m)}{m(1-q(m))}-1}
        {\sqrt{\frac{q(m)(k-m)}{m(1-q(m))}}}
        = 
        \frac{k\delta}{\sqrt{m(1-q(m))q(m)(k-m)}}
        .
    \end{equation*}
    It follows that
    \begin{equation*}
        \left(\frac{L}{\alpha}\right)^2
        <
        \frac{(k\delta)^2}{m(k-m)q(m)(1-q(m))},
    \end{equation*}
    and therefore $h'(m)>0$. Hence $S_k(c)=h(k-1)$, as claimed.
\end{proof}

\section{The \texorpdfstring{$\Werner_\gamma$}{Q gamma} tensor}
\label{sec:werner-gamma}

Here we consider the $4$-tensor $\Werner_\gamma \in \mathbb C^{2}\ot \mathbb C^{2} \ot \mathbb C^{2} \ot \mathbb C^{2}$ which can be defined in terms of four $2$-tensors on $\mathbb C^{2} \ot \mathbb C^{2}$:
\begin{align*}
    \Phi_{\pm} = \frac{e_1\ot e_1\pm e_2\ot e_2}{\sqrt{2}}, \qquad
    \Psi_{\pm} = \frac{e_1 \ot e_2 \pm e_2 \ot e_0}{\sqrt{2}}.
\end{align*}
For $\gamma \in [0,1]$, the tensor $\Werner_\gamma$ is of the form
\begin{align}
    \Werner_\gamma 
    &\coloneqq
    \sqrt{1-\gamma}\left(\Psi_{-}^{AB}\ot\Psi_{-}^{CD}\right)+\sqrt{\frac{\gamma}{3}}\left(\Phi_{+}^{AB}\ot\Phi_{+}^{CD}+\Phi_{-}^{AB}\ot\Phi_{-}^{CD}+\Psi_{+}^{AB}\ot\Psi_{+}^{CD}\right).
\end{align}
This unit tensor can understood as a purification of the so-called \textit{Werner state} \cite{werner1989quantum}
\begin{align}
    \rho_\gamma 
    &=  (1-\gamma)P_{\vee^2\mathbb C^2} + \frac{\gamma}{3} P_{\wedge^2\mathbb C^2}.
\end{align}
which is a positive semidefinite matrix on $\mathbb C^{2}\ot \mathbb C^{2}$ for $\gamma \in [0,1]$.
For any singleton bipartition, e.g. $A|BCD$, we have uniform eigenvalues for the corresponding marginal,
\[
    \mathrm{eig}_{A|BCD}(\Werner_\gamma) = \left(\frac{1}{2},\frac{1}{2}\right),
\]
while for the $AB|CD$ bipartition, we have non-uniform eigenvalues
\[
    \mathrm{eig}_{AB|CD}(\Werner_\gamma) = \left(1-\gamma,\frac{\gamma}{3},\frac{\gamma}{3},\frac{\gamma}{3}\right).
\]
The $AB|CD$ entropy of $\Werner_\gamma$ is therefore
\[
    H_{AB|CD}(\Werner_\gamma)=-\gamma \log_2 \frac{\gamma}{3} -(1-\gamma)\log_2 (1-\gamma).
\]
We use the determinant bound from \cref{cor:Edetab bound,sec:max_ent_with_det}.
First note that
\[
  c_{\Werner_\gamma} = \frac{\abs{\det((\Werner_\gamma)_{AB|CD})}^2}{\capa(\Werner_\gamma)^{8}} = (1-\gamma)\left(\frac{\gamma}{3}\right)^3,
\]
since~$\Werner_\gamma$ has uniform one-body marginals and hence has~$\capa(\Werner_\gamma)=\norm{\Werner_\gamma}=1$ (\cref{cor:kempf-ness-tensor}).
We can compute $\tilde E^{\det}_{AB}(c_{\Werner_\gamma})$ using 
\cref{prop:Edetab-exact-value} and determine 
\[
    \gamma \in \left[\frac{3}{4},1\right] \implies  E_{AB}(\Werner_\gamma) \leq \tilde E^{\det}_{AB}(c_{\Werner_\gamma}) = H_{AB|CD}(\Werner_\gamma).
\]
The interval $\frac{3}{4} \leq \gamma \leq 1$ corresponds to the regime where $\frac{\gamma}{3}$ is a larger eigenvalue than $1-\gamma$.
The last equality follows from~\cref{eq:Skc value}.
Since $H_{AB|CD}(\Werner_\gamma) \leq E_{AB}(\Werner_\gamma)$ by definition, we conclude $E_{AB}(\Werner_\gamma) = H_{AB|CD}(\Werner_\gamma)$ holds for all $\gamma \in \left[\frac{3}{4},1\right]$.

\section{Character formulas for isotypic projectors}
\label{sec:char_forms}

The definition of the upper quantum functional, \cref{def:upper qfunc}, on a tensor $\psi$, involves evaluating a product of projection operators, $\prod_{b \in \supp \theta} P_{\lambda^{(b)}}^{V_b}$, applied to the $n$-th Kronecker power $\psi^{\kron n}$.
The purpose of this section is to describe formulas for these projection operators in terms of characters, $\chi_{\lambda}$ of the symmetric group $S_n$.

For any representation $\eta\colon S_n \to \mathrm{End}(W)$ on a finite dimensional complex vector space $W$, the operator
\begin{equation}
    P_{\lambda}^{W} \coloneq  \frac{\dim([\lambda])}{n!} \sum_{\pi \in S_n} \overline{\chi_{\lambda}(\pi)} \eta(\pi)
    \label{eq:proj_char_form}
\end{equation}
is the projection operator onto the isotypic subspace $[\lambda] \otimes \mathrm{Hom}([\lambda], W) \xhookrightarrow{} W$ of $W$ \cite[Sec.~2.4]{fulton2013representation}.
Here $\chi_{\lambda}$ denotes the character of the irreducible representation $\pi_{\lambda}\colon S_n \to \mathrm{End}([\lambda])$ of the symmetric group $S_n$ labelled by the Young diagram $\lambda$.
Note that $\overline{\chi_{\lambda}(\pi)} = \chi_{\lambda^{*}}(\pi)=\chi_{\lambda}(\pi^{-1})$ holds for all groups, but additionally $\chi_{\lambda}(\pi^{-1}) = \chi_{\lambda}(\pi)$ holds for the symmetric group as permutations are conjugate to their inverses.

For the purposes of the upper quantum functional, the relevant representations of the symmetric group are the tensor-permutation representations $\tau \colon S_n \to \mathrm{End}(V^{\otimes n})$ which permute the tensor factors of $V$.
Given a bipartition $b = \{S, \overline{S}\}$ of $[k]$, let $\tau_{S}$ (resp. $T_{\overline{S}}$) denote the tensor permutation representation associated to $V_S$ (resp. $V_{\overline{S}}$).
After applying \cref{eq:proj_char_form} to the definition $P_{\lambda}^{V_{b}} \coloneqq P_\lambda^{V_S} P_{(n)}^{V_{[k]}}$ we obtain the formula,
\begin{equation}
    P_{\lambda}^{V_{b}} = \frac{\dim([\lambda])}{(n!)^{2}}\sum_{\pi, \sigma \in S_n} \overline{\chi_{\lambda}(\pi)}  \tau_{S}(\pi \sigma) \otimes \tau_{\overline{S}}(\sigma)
    \label{eq:proj-restricted-char-form}
\end{equation}
Using this formula for $P_{\lambda}^{V_{b}}$ we are able to evaluate any product of projectors $\prod_{b \in \supp \theta} P_{\lambda^{(b)}}^{V_b}$ relevant for the upper quantum functional $F^{\theta}$.

\section{Separations for distributions with non-laminar support}\label{sec:crossing}
We show that the upper and lower quantum functionals can be separated already for very simple distributions with non-laminar support\footnote{We follow the implicit suggestion in the original paper~\cite[Rem. 3.15]{cvz23universal} defining the upper quantum functional with respect to a~\emph{fixed} ordering of the projectors.}.
Specifically, for $\theta=(\theta_{AB|CD},\theta_{BC|AD})$ with~$\theta_{AB|CD}, \theta_{BC|AD} \neq 0$, we show that the rank-$2$ unit tensor~$\unit{2}$ on~$4$ systems~$A,B,C,D$ satisfies $E^\theta(\unit{2})>E_\theta(\unit{2})$.

\begin{theorem}\label{thm:fixed-order-crossing-separation}
    For any $0<\theta_{AB|CD}<1$, define $\theta_{BC|AD}=1-\theta_{AB|CD}$ and $\theta=(\theta_{AB|CD},\theta_{BC|AD})$.
    Let
    \[
        \unit{2} = \sum_{i=1}^2 e_i \ot e_i \ot e_i \ot e_i.
    \]
    Then we have
    \[
        E_\theta(\unit{2}) \leq 1 < E^\theta(\unit{2}),
    \]
    where for the definition of the logarithmic upper quantum functional we fix the order of operators as follows:
    \[
    \begin{split}
        E^\theta(\psi) = \sup \{& \theta_{AB|CD} H\big(\frac{\lambda_{AB|CD}}{n}\big) + \theta_{BC|AD} H\big(\frac{\lambda_{BC|AD}}{n}\big) : \\
        & n \geq 1, \, \lambda_{AB|CD}, \lambda_{BC|AD} \vdash n, \, P^{{AB|CD}}_{\lambda_{AB|CD}}P^{BC|AD}_{\lambda_{BC|AD}} \psi^{\ot n} \neq 0 \}. 
    \end{split}
    \]
\end{theorem}
To prove \cref{thm:fixed-order-crossing-separation}, we upper bound the logarithmic lower quantum functional over $\unit{2}$ for $\theta \in \dists$, and then lower bound the logarithmic upper quantum functional over it for the specific distributions considered in the theorem.
\begin{claim}
    For any~$\theta \in \dists$, 
    \[
        E_\theta(\unit{2}) \leq \log_2(2) = 1.
    \]
\end{claim}
\begin{proof}
    For any~$\emptyset \neq S \subsetneq \{A,B,C,D\}$, the flattening rank~$\rank_S(\unit{2})$ is~$2$.
    Since flattening ranks are degeneration monotone, if~$\unit{2} \degengeq \psi$, then~$\rank_S(\psi) \leq 2$.
    Therefore~$H_S(\psi) \leq 1$, and also~$H_\theta(\psi) \leq 1$.
\end{proof}

\begin{claim}\label{cl:crossing-not-normalised}
    Let $\theta=(\theta_{AB|CD},\theta_{BC|AD})$ for  $0<\theta_{AB|CD}<1$ and $\theta_{BC|AD}=1-\theta_{AB|CD}$. Then 
    \[E^\theta(\unit{2})\geq \frac{3}{2}\theta_{AB|CD}+\theta_{BC|AD}>1\]
\end{claim}
\begin{proof}
    The lower bound follows from a particular pair of Young diagrams $\lambda_{AB|CD} = \vcenter{\hbox{\ydiagram{2,1,1}}}$ and $\lambda_{BC|AD} = \vcenter{\hbox{\ydiagram{2,2}}}$ such that $P^{{AB|CD}}_{\lambda_{AB|CD}}P^{BC|AD}_{\lambda_{BC|AD}} \psi^{\ot n} \neq 0$.
    Using the character formula \cref{eq:proj-restricted-char-form} for $P^{AB|CD}_{\ydiagram{2,1,1}}$ and $P^{BC|AD}_{\ydiagram{2,2}}$,
    we conclude the following does not vanish:
    \[ P^{AB|CD}_{\ydiagram{2,1,1}}P^{BC|AD}_{\ydiagram{2,2}}\unit{2}^{\otimes 4} \neq 0.
    \qedhere
    \]
    As the entropies of the normalizations of $\vcenter{\hbox{\ydiagram{2,1,1}}}$ and $\vcenter{\hbox{\ydiagram{2,2}}}$ are respectively $H((\frac{2}{4},\frac{1}{4},\frac{1}{4})) = \frac{2}{4}\log_2 \frac{4}{2} + 2 \cdot\frac{1}{4}\log_2 \frac{4}{1} = \frac{3}{2}$ and $H((\frac{2}{4},\frac{2}{4})) = 2 \cdot \frac{2}{4} \log_2 \frac{4}{2} = 1$, the claim follows.
\end{proof}
\begin{remark}
    Note that the ordering of the projection operators one considers in defining the upper functional is crucial here because $P^{AB|CD}_{\ydiagram{2,1,1}}\unit{2}^{\otimes 4} = 0$ (as the $AB|CD$ flattening rank of $\unit{2}$ is $2$), and thus
    \[        P^{BC|AD}_{\ydiagram{2,2}}P^{AB|CD}_{\ydiagram{2,1,1}}\unit{2}^{\otimes 4} = 0, \qquad
    P^{AB|CD}_{\ydiagram{2,1,1}}P^{BC|AD}_{\ydiagram{2,2}}\unit{2}^{\otimes 4} \neq 0.
    \]
\end{remark}

\bibliographystyle{alphaurl}
\bibliography{ref}

\end{document}